\documentclass[11pt,a4paper,leqno]{amsart}
\usepackage{amsmath,amsfonts,amssymb}
\usepackage[alphabetic]{amsrefs}
\usepackage{hyperref}
\usepackage{graphicx, color}
\usepackage{pictexwd,dcpic,epsf}
\usepackage{enumerate}
\usepackage{fullpage}
\usepackage{xspace}
\usepackage[plain,vlined,english,norelsize]{algorithm2e} 

\usepackage{epsfig}
\usepackage{color}
\usepackage{amsmath}
\usepackage{amssymb}
\newtheorem{lemma}{Lemma}
\newtheorem{proposition}[lemma]{Proposition}
\newtheorem{theorem}[lemma]{Theorem}
\newtheorem{corollary}[lemma]{Corollary}

\theoremstyle{definition}
\newtheorem{example}[lemma]{Example}
\newtheorem{remark}[lemma]{Remark}

\newtheorem{conjecture}[lemma]{Conjecture}
\newtheorem{claim}[lemma]{Claim}

\newcommand{\covectors}{\ensuremath{\mathcal{L}}}

\newcommand{\conv}{\ensuremath{{\rm conv}}}
\newcommand{\gate}{\ensuremath{{\rm gate}}}

\newcommand{\uX}{\ensuremath{\underline{\mathcal{X}}}\xspace}
\newcommand{\oX}{\ensuremath{\overline{\mathcal{X}}}\xspace}

\newcommand{\OM}{\mathrm{OM}}
\newcommand{\AOM}{\mathrm{AOM}}

\renewcommand{\SS}{\ensuremath{\mathcal{S}}}
\newcommand{\uparr}{\ensuremath{\hspace{3pt}\uparrow\hspace{-2pt}}}

\DeclareMathOperator{\vcd}{VC-dim}
\renewcommand{\SS}{\ensuremath{\mathcal{S}}}

\usepackage{graphicx}

\title{Two-dimensional partial cubes\thanks{The work on this paper was supported by ANR project DISTANCIA (ANR-17-CE40-0015).}}

\author[V.\ Chepoi]{Victor Chepoi}
\address{Laboratoire d'Informatique et Systh\`emes, Aix-Marseille Universit\'e and CNRS, Facult\'e des Sciences de Luminy, F-13288 Marseille Cedex 9, France}
\email{victor.chepoi@lis-lab.fr}

\author[K.\ Knauer]{Kolja Knauer}
\address{Laboratoire d'Informatique et Systh\`emes, Aix-Marseille Universit\'e and CNRS, Facult\'e des Sciences de Luminy, F-13288 Marseille Cedex 9, France\\Departament de Matem\`atiques i Inform\`atica,
Universitat de Barcelona (UB), Barcelona, Spain}
\email{kolja.knauer@lis-lab.fr}

\author[M.\ Philibert]{Manon Philibert}
\address{Laboratoire d'Informatique et Systh\`emes, Aix-Marseille Universit\'e and CNRS, Facult\'e des Sciences de Luminy, F-13288 Marseille Cedex 9, France}
\email{manon.philibert@lis-lab.fr}

\begin{document}

\maketitle

\begin{abstract}
We investigate the structure of two-dimensional partial cubes, i.e., of
isometric subgraphs of hypercubes whose vertex set defines a set family of VC-dimension at most 2.
Equivalently, those are the partial cubes which are not contractible to the 3-cube $Q_3$ (here contraction
means contracting the edges corresponding to the same coordinate
of the hypercube). We show that our graphs can be obtained from two types of
combinatorial cells (gated cycles and gated full subdivisions of complete graphs)
via amalgams. The cell structure of two-dimensional partial cubes enables us to
establish a variety of results. In particular, we prove that all partial cubes of VC-dimension 2
can be extended to ample aka lopsided partial cubes of VC-dimension 2, yielding that the set families
defined by such graphs satisfy the sample compression conjecture by Littlestone and Warmuth (1986) in a strong sense.
The latter is a central conjecture of the area of computational machine learning, that is far from being solved even for general set systems of VC-dimension 2.
Moreover, we point out relations to tope graphs of COMs of low rank and region graphs of pseudoline arrangements.
\end{abstract}

\tableofcontents

\section{Introduction}
Set families are fundamental objects in combinatorics, algorithmics,
machine learning, discrete geometry, and combinatorial optimization. The Vapnik-Chervonenkis
dimension (the \emph{VC-dimension} for short) $\vcd(\SS)$  of a set family $\SS\subseteq 2^U$ is
the size of a largest subset of $X\subseteq U$ which can be \emph{shattered}
by $\SS$ \cite{VaCh}, i.e., $2^{X}=\{X\cap S: S\in\mathcal{S}\}$.
Introduced in
statistical learning by  Vapnik and Chervonenkis \cite{VaCh}, the VC-dimension was adopted in the above areas
as complexity measure and as a combinatorial dimension of $\SS$.
Two important inequalities relate a set family $\SS\subseteq 2^{U}$ with its VC-dimension. The first one, the
\emph{Sauer-Shelah lemma}  \cite{Sauer,Shelah} establishes that if $|U|=m$, then the number of sets in a set family
$\SS\subseteq 2^{U}$ with VC-dimension  $d$ is upper bounded by  $\binom{m}{\leq d}$. The second stronger
inequality, called the \emph{sandwich lemma}, proves that $|\SS|$ is sandwiched between the number
of \emph{strongly shattered} sets (i.e., sets $X$ such that $\SS$ contains an $X$-cube, see Section \ref{OM-COM-AMP})
and the number of shattered sets \cite{AnRoSa,BoRa,Dr,Pa}. The set families
for which the Sauer-Shelah bounds are tight are called {\it maximum families}  \cite{GaWe,FlWa} and the set families
for which the upper bounds in the sandwich lemma are tight are called {\it ample, lopsided, and extremal families}
\cite{BaChDrKo,BoRa,La}. Every  maximum  family  is  ample,  but  not  vice versa.

To take a graph-theoretical point of view on set families, one considers the subgraph
$G(\SS)$ of the hypercube $Q_m$ induced by the subsets of $\SS\subseteq 2^{U}$. (Sometimes $G(\SS)$
is called the {\it 1-inclusion graph} of $\SS$ \cite{Hau,HaLiWa}.) Each edge of $G(\SS)$ corresponds to
an element of $U$. Then analogously to edge-contraction and minors in graph theory, one can
consider the operation of simultaneous contraction of all edges of $G(\SS)$ defined by the same element $e\in U$. The resulting
graph is the 1-inclusion graph $G(\SS_e)$ of the set family $\SS_e\subseteq 2^{U\setminus \{ e\}}$ obtained by identifying
all pairs of sets of $\SS$ differing only in $e$. Given $Y\subseteq U$, we call the set family $\SS_Y$ and its 1-inclusion
graph $G(\SS_Y)$ obtained from $\SS$ and $G(\SS)$ by successively contracting the edges labeled by the elements of $Y$
the {\it Q-minors} of $\SS$ and $G(\SS)$. Then $X\subseteq U$ is shattered by $\SS$ if and only if the Q-minor
$G(\SS_{U\setminus X})$ is a full cube. Thus, the cubes play the same role for Q-minors as the complete graphs for classical
graph minors.

To take a metric point of view on set families, one restricts to set families whose 1-inclusion graph satisfies
further properties. The typical property here is that the 1-inclusion graph $G(\SS)$ of $\SS$ is an isometric (distance-preserving)
subgraph  of the hypercube $Q_m$. Such graphs 
are called {\it partial cubes}. Partial cubes
can be characterized in a pretty and efficient way \cite{Dj} and can be recognized in quadratic time \cite{Epp}.
Partial cubes comprise many important and complex graph classes occurring in metric graph theory and
initially arising in completely different areas of research such as geometric group theory, combinatorics, discrete geometry,
and media theory (for a comprehensive presentation of partial cubes and their classes, see the survey \cite{BaCh_survey} and
the books \cite{DeLa,HaImKl,Ov1}). For example, 1-inclusion graphs of ample families (and thus of maximum families) are partial
cubes \cite{BaChDrKo,La} (in view of this, we will call
such graphs {\it ample partial cubes} and {\it maximum partial cubes}, respectively).
Other important examples comprise median graphs (aka 1-skeletons of CAT(0) cube
complexes \cite{Ch_CAT,Sa}) and, more generally, 1-skeletons of CAT(0) Coxeter zonotopal complexes~\cite{HaPa},
the tope graphs of oriented matroids (OMs)~\cite{BjLVStWhZi}, of affine oriented matroids (AOMs)~\cite{KnMa}, and of lopsided
sets (LOPs)~\cite{KnMa,La}, where the latter coincide with ample partial cubes (AMPs).
More generally, tope graphs of complexes of oriented matroids (COMs)~\cite{BaChKn,KnMa} capture all of the above.
Other classes of graphs defined by distance or convexity properties turn
out to be partial cubes:   bipartite cellular graphs (aka bipartite graphs with totally decomposable metrics)~\cite{BaCh_cellular},
bipartite Pasch \cite{Ch_thesis,Ch_separation} and bipartite Peano \cite{Po_Peano}
graphs, netlike graphs~\cite{Po1}, and hypercellular graphs \cite{ChKnMa}.

Many mentioned classes of partial cubes can be characterized via forbidden $Q$-minors; in case of partial cubes, $Q$-minors
are endowed with a second operation called \emph{restriction} and are
called {\it partial cube minors}, or {\it pc-minors} \cite{ChKnMa}. The class of partial cubes is closed under pc-minors.
Thus, given a set $G_1,G_2,\ldots,G_n$ of partial cubes, one considers the set ${\mathcal F}(G_1,\ldots,G_n)$ of all partial cubes not having any
of $G_1,G_2,\ldots,G_n$ as a pc-minor. Then  $\mathcal{F}(Q_2)$ is the class of  trees, $\mathcal{F}(P_3)$ is the class
of hypercubes, and $\mathcal{F}(K_2\square P_3)$ consists of bipartite cacti~\cite[page 12]{Ma}. Other obstructions lead to
more interesting classes, e.g., almost-median graphs ($\mathcal{F}(C_6)$~\cite[Theorem 4.4.4]{Ma}),
hypercellular graphs ($\mathcal{F}(Q_3^-)$~\cite{ChKnMa}), median graphs ($\mathcal{F}(Q_3^-, C_6)$~\cite{ChKnMa}),
bipartite cellular graphs ($\mathcal{F}(Q_3^-, Q_3)$~\cite{ChKnMa}), 
rank two COMs ($\mathcal{F}(SK_4, Q_3)$~\cite{KnMa}), and two-dimensional ample graphs ($\mathcal{F}(C_6, Q_3)$~\cite{KnMa}).
Here $Q_3^-$ denotes the 3-cube $Q_3$ with one vertex removed and $SK_4$ the full subdivision of $K_4$, see Figure~\ref{fig:COMobstructions}. Bipartite Pasch graphs have been
characterized in \cite{Ch_thesis,Ch_separation} as partial cubes excluding 7 isometric subgraphs of $Q_4$ as  pc-minors.

Littlestone and Warmuth~\cite{LiWa} introduced the sample compression technique for deriving generalization bounds in machine learning.
Floyd and Warmuth~\cite{FlWa} asked whether any set family $\SS$ of VC-dimension $d$ has a sample compression scheme of size~$O(d)$.
This question remains one of the oldest open problems in computational machine learning. It was recently shown in \cite{MoYe} that
labeled compression schemes of size $O(2^d)$ exist. Moran and Warmuth  \cite{MoWa} designed labeled  compression schemes of size $d$ for ample families.
Chalopin et al. \cite{ChChMoWa} designed (stronger) unlabeled compression schemes of size $d$ for maximum families and characterized such schemes for
ample families via unique sink orientations of their 1-inclusion graphs. For ample families of VC-dimension 2 such unlabeled
compression schemes exist because they admit corner peelings \cite{ChChMoWa,MeRo}. In view of this, it was noticed in \cite{RuRuBa}
and \cite{MoWa} that  the original sample compression conjecture of \cite{FlWa} would be solved if {\it  one can show that any set family $\SS$
of VC-dimension $d$ can be extended to an ample (or maximum) partial cube of VC-dimension $O(d)$ or can be covered by $exp(d)$ ample
partial cubes of VC-dimension $O(d)$.}
These questions are already nontrivial for set families of VC-dimension 2.

In this paper, we investigate the first question for partial cubes of VC-dimension 2, i.e., the class $\mathcal{F}(Q_3)$, that we will simply call {\it two-dimensional partial cubes}.
We show that two-dimensional partial cubes can be extended to ample partial cubes of VC-dimension 2 -- a property that is not shared by general set families of VC-dimension 2.
In relation to this result,
we establish that all two-dimensional partial cubes can be obtained via amalgams from two types of combinatorial cells: maximal full subdivisions of complete graphs and
convex cycles not included in such subdivisions. We show that all such cells are gated subgraphs.
On the way, we detect a variety of other structural
properties of two-dimensional partial cubes. Since two-dimensional partial cubes are very natural from the point of view of pc-minors and generalize previously studied classes such as bipartite cellular graphs~\cite{ChKnMa}, we consider these results of independent interest also from this point of view. In particular, we point out relations to tope graphs of COMs of low rank and region graphs of pseudoline arrangements. See Theorem~\ref{characterization} for a full statement of our results on two-dimensional partial cubes. Figure \ref{fig:two-dim-pc} presents an example of a two-dimensional partial cube which we further use as a running example. We also provide two characterizations of partial cubes of VC-dimension $\le d$ for any $d$ (i.e., of the class ${\mathcal F}(Q_{d+1})$) via hyperplanes
and isometric expansions. However, \emph{understanding the structure of graphs from
${\mathcal F}(Q_{d+1})$ with $d\ge 3$ 
remains a challenging open question.}

%

\begin{figure}[htb]
\centering
\includegraphics[width=.35\textwidth]{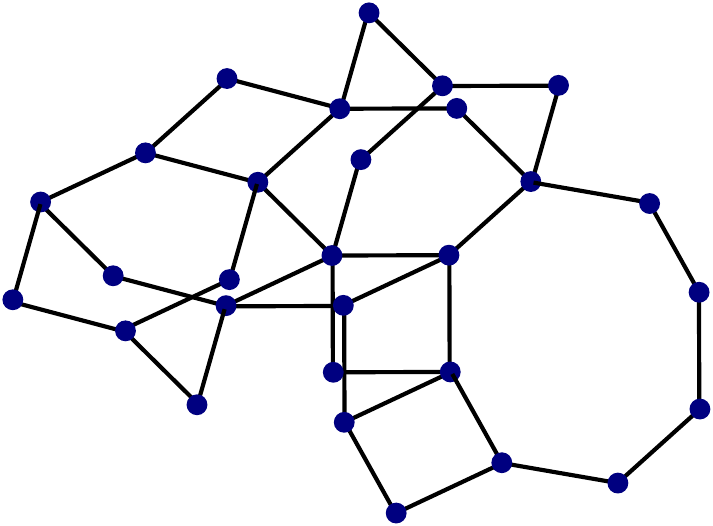}
\caption{A two-dimensional partial cube $M$}
\label{fig:two-dim-pc}
\end{figure}

\section{Preliminaries}


\subsection{Metric subgraphs and partial cubes}
All graphs $G=(V,E)$ in this paper  are finite, connected,
and simple. 
The {\it distance} $d(u,v):=d_G(u,v)$ between two vertices $u$ and $v$ is the length
of a shortest $(u,v)$-path, and the {\it interval} $I(u,v)$
between $u$ and $v$ consists of all vertices on shortest $(u,v)$-paths:
$I(u,v):=\{ x\in V: d(u,x)+d(x,v)=d(u,v)\}.$
An induced subgraph $H$ of $G$ is {\it
isometric} if the distance between any pair of vertices in $H$ is
the same as that in $G.$
An induced subgraph of $G$ (or the corresponding vertex set $A$)
is called {\it convex} if it includes the interval of $G$ between
any two of its vertices.  Since the intersection of convex subgraphs
is convex, for every subset $S\subseteq V$ there exists the
smallest convex set $\conv(S)$ containing
$S$, referred to as the {\it convex hull} of $S$.
A subset $S\subseteq V$ or the
subgraph $H$ of $G$ induced by $S$ is called {\it gated} (in $G$)~\cite{DrSch}
if for every vertex $x$ outside $H$ there exists a vertex $x'$
(the {\it gate} of $x$) in $H$ such that each vertex $y$ of $H$ is
connected with $x$ by a shortest path passing through the gate
$x'$. It is easy to see that if $x$ has a gate in $H$, then it is unique
and that gated sets are convex. 
Since the intersection of gated subgraphs
is gated, for every subset $S\subseteq V$ there exists the
smallest gated set $\gate( S)$ containing
$S,$ referred to as the {\it gated hull} of $S$.

A graph $G=(V,E)$ is
{\it isometrically embeddable} into a graph $H=(W,F)$ if there
exists a mapping $\varphi : V\rightarrow W$ such that $d_H(\varphi
(u),\varphi (v))=d_G(u,v)$ for all vertices $u,v\in V$,
 i.e., $\varphi(G)$ is  an isometric subgraph of $H$. A graph $G$ is called a {\it partial cube} if it admits an isometric embedding into some hypercube
$Q_m$. 
For an edge $e=uv$ of $G$, let $W(u,v)=\{ x\in V: d(x,u)<d(x,v)\}$.  
For an edge $uv$, the sets $W(u,v)$ and $W(v,u)$ are called {\it complementary halfspaces} of $G$.

\begin{theorem} \cite{Dj} \label{Djokovic}  A graph $G$ is a partial cube if and only if $G$ is bipartite and for any edge $e=uv$
the sets $W(u,v)$ and $W(v,u)$ are convex.
\end{theorem}

To establish an isometric embedding of $G$ into a hypercube,  Djokovi\'{c}~\cite{Dj}
introduced the following binary relation $\Theta$ (called \emph{Djokovi\'{c}-Winkler relation}) on the edges of $G$:  for two edges $e=uv$ and $e'=u'v'$ we set $e\Theta e'$ if and only if
$u'\in W(u,v)$ and $v'\in W(v,u)$. Under the conditions of the theorem, $e\Theta e'$ if and only if
$W(u,v)=W(u',v')$ and $W(v,u)=W(v',u')$, i.e. $\Theta$ is an equivalence relation. Let $E_1,\ldots,E_m$ be the equivalence classes
of $\Theta$ and let $b$ be an arbitrary vertex taken as the basepoint of $G$.
For a $\Theta$-class $E_i$,
let $\{ G^-_i,G^+_i\}$ be the pair of complementary
convex halfspaces  of $G$ defined by setting $G^-_i:=G(W(u,v))$ and $G^+_i:=G(W(v,u))$ for an arbitrary edge $uv\in E_i$ such that $b\in G^-_i$.
Then the isometric embedding $\varphi$ of $G$ into the $m$-dimensional hypercube $Q_m$ is obtained by setting
$\varphi(v):=\{ i: v\in G^+_i\}$ for any vertex $v\in V$.  Then $\varphi(b)=\varnothing$ and for any two vertices $u,v$ of $G$, $d_G(u,v)=|\varphi (u)\Delta \varphi(v)|.$

The bipartitions $\{ G^-_i,G^+_i\}, i=1,\ldots,m,$ can be canonically defined for all subgraphs $G$ of the hypercube $Q_m$, not only for partial cubes. Namely,
if $E_i$ is a class of parallel edges of $Q_m$, then removing the edges of $E_i$ from $Q_m$ but leaving their end-vertices, $Q_m$ will be divided into two $(m-1)$-cubes $Q'$ and $Q''$.
Then $G^-_i$ and $G^+_i$ are the intersections of $G$ with $Q'$ and $Q''$.

For a $\Theta$-class $E_i$, 
the {\it boundary} $\partial G^-_i$ of the halfspace $G^-_i$ consists of all vertices of $G^-_i$ having a neighbor in $G^+_i$
($\partial G^+_i$ is defined analogously). Note that $\partial G^-_i$ and $\partial G^+_i$ induce isomorphic subgraphs (but not
necessarily isometric) of $G$. Figure \ref{fig:M-Theta-class}(a) illustrates a $\Theta$-class $E_i$ of the two-dimensional partial cube $M$,
the halfspaces $M_i^-, M_i^+$ and their boundaries $\partial M_i^-,\partial M_i^+$.

\begin{figure}[htb]
\centering
\includegraphics[width=0.80\textwidth]{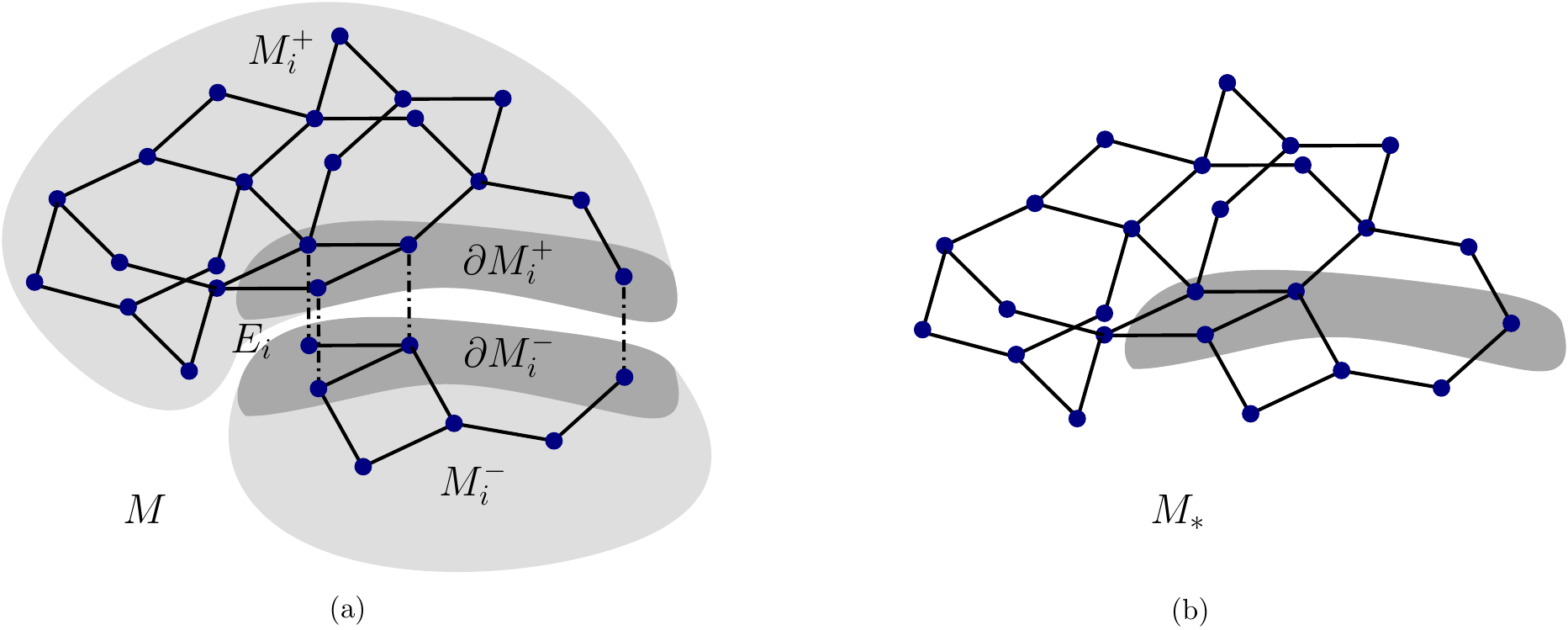}
\caption{(a) The halfspaces and their boundaries defined by a $\Theta$-class $E_i$ of $M$. (b) The two-dimensional partial cube $M_*=\pi_i(M)$ obtained from $M$ by contracting $E_i$.}
\label{fig:M-Theta-class}
\end{figure}

An \emph{antipode} of a vertex $v$ in a partial cube $G$ is a vertex $-v$ such that $G=\conv(v,-v)$. Note that in partial cubes the antipode is unique and $\conv(v,-v)$
coincides with the interval $I(v,-v)$. A partial cube $G$ is \emph{antipodal} if all its vertices have antipodes. A partial cube $G$ is said to be \emph{affine} if there
is an antipodal partial cube $G'$, such that $G$ is a halfspace of $G'$.

\subsection{Partial cube minors}\label{minors}
Let $G$ be a partial cube, isometrically embedded in the hypercube $Q_m$.
 For a $\Theta$-class $E_i$ of $G$,  an {\it elementary restriction} consists of taking one of the
 complementary halfspaces $G^-_i$ and $G^+_i$. 
More generally, a {\it restriction} is a subgraph of $G$ induced by the
intersection of a set of (non-complementary) halfspaces of $G$. Such an intersection is a convex subgraph of $G$, thus a partial cube. Since any convex subgraph of a partial cube $G$ is the intersection of halfspaces \cite{AlKn,Ch_thesis}, the restrictions of $G$ coincide with the convex subgraphs of $G$.

For a $\Theta$-class $E_i$, we say that the graph  $\pi_i(G)$ obtained from
$G$ by contracting the edges of $E_i$ is an ($i$-){\it contraction} of $G$; for an illustration, see Figure \ref{fig:M-Theta-class}(b). For a vertex $v$ of $G$, we will denote by $\pi_i(v)$ the image of $v$ under the $i$-contraction, i.e., if $uv$ is an
edge of $E_i$, then $\pi_i(u)=\pi_i(v)$, otherwise $\pi_i(u)\ne \pi_i(v)$. We will apply $\pi_i$ to subsets $S\subset V$, by setting $\pi_i(S):=\{\pi_i(v): v\in S\}$. In particular we denote the $i$-{\it contraction} of $G$ by $\pi_i(G)$.
From the proof of the first part of ~\cite[Theorem 3]{Ch_hamming} it easily follows that $\pi_i(G)$ is an isometric subgraph of $Q_{m-1}$, thus the class of partial cubes is closed under contractions.
Since edge contractions in graphs commute, if $E_i,E_j$ are two distinct $\Theta$-classes, then $\pi_j(\pi_i(G))=\pi_i(\pi_j(G))$. Consequently, for a set $A$ of $k$ $\Theta$-classes, we can denote
by $\pi_A(G)$ the isometric subgraph
of $Q_{m-k}$ obtained from $G$ by contracting the equivalence classes of edges from $A$.

Contractions and restrictions commute in partial cubes \cite{ChKnMa}. Consequently,  any set of restrictions and any set of contractions of a
partial cube $G$ provide the same result, independently of the order in which we perform them. 
The resulting graph $G'$ is  a partial cube, and $G'$ is called a {\it partial cube minor} (or {\it pc-minor}) of $G$. For a partial cube $H$ we denote by ${\mathcal F}(H)$ the class of all partial cubes
not having $H$ as a pc-minor. In this paper we investigate the class ${\mathcal F}(Q_3)$.

\begin{figure}[htb]
\centering
\includegraphics[width=0.90\textwidth]{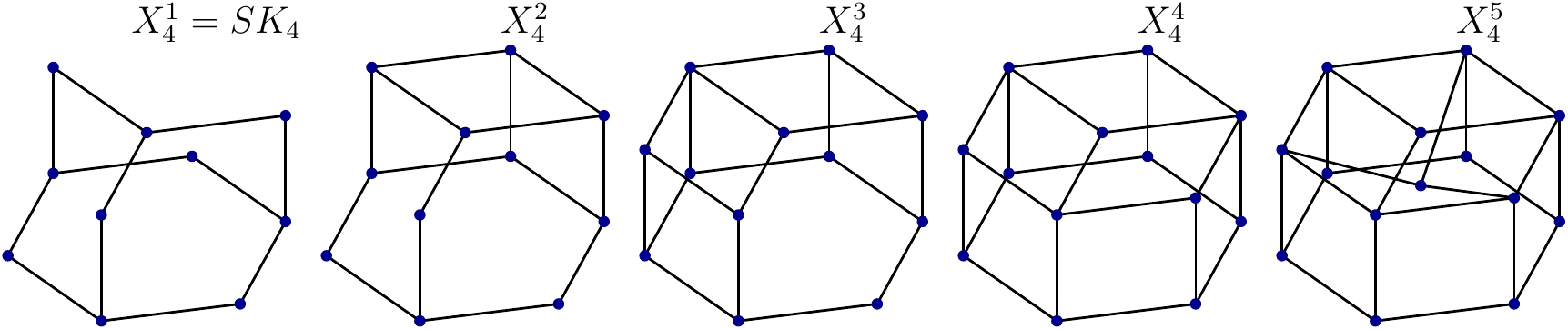}
\caption{The excluded pc-minors of isometric dimension $\le 4$ for COMs.}
\label{fig:COMobstructions}
\end{figure}

With the observation that a convex subcube of a partial cube can be obtained by contractions as well, the proof of the following lemma is straightforward.

\begin{lemma} \label{VCdim_d} A partial cube $G$ belongs to ${\mathcal F}(Q_{d+1})$ if and only if $G$ has VC-dimension $\le d$.
\end{lemma}

Let $G$ be a partial cube and $E_i$ be a $\Theta$-class of $G$. Then  $E_i$ {\it crosses} a convex subgraph
$H$ of $G$ if $H$ contains an edge $uv$ of $E_i$ and $E_i$ {\it osculates} $H$ if $E_i$ does not cross $H$ and there exists
an edge $uv$ of $E_i$ with $u\in H$ and $v\notin H$. Otherwise, $E_i$ is {\it disjoint} from $H$. The following results summarize
the properties of contractions of partial cubes established in \cite{ChKnMa} and \cite{KnMa}:

\begin{lemma} \label{contraction-ChKnMa} Let $G$ be a partial cube and $E_i$ be a $\Theta$-class of $G$.

\begin{itemize}
\item[(i)] \cite[Lemma 5]{ChKnMa} If $H$ is a convex subgraph of $G$ and $E_i$ crosses or is disjoint from $H$,
then $\pi_i(H)$ is also a convex subgraph of $\pi_i(G)$;
\item[(ii)] \cite[Lemma 7]{ChKnMa} If $S$ is a subset of vertices of $G$, then $\pi_i(\conv(S))\subseteq \conv(\pi_i(S))$.
If $E_i$ crosses $S$, then $\pi_i(\conv(S))= \conv(\pi_i(S))$;
\item[(iii)] \cite[Lemma 10]{ChKnMa} If $S$ is a gated subgraph of $G$, then $\pi_i(S)$ is a  gated subgraph of $\pi_i(G)$.
\end{itemize}
\end{lemma}

\begin{lemma}\cite{KnMa}\label{antipodal-contraction}
 Affine and antipodal partial cubes are closed under contractions.
\end{lemma}

\subsection{OMs, COMs, and AMPs}\label{OM-COM-AMP} In this subsection, we recall the definitions of oriented matroids, complexes of oriented matroids, and ample set families.

\subsubsection{OMs: oriented matroids}\label{sub:OM} Co-invented by Bland $\&$ Las Vergnas~\cite{BlLV} and Folkman $\&$ Lawrence~\cite{FoLa}, and further investigated
by many other authors,
oriented matroids  represent a unified combinatorial theory of orientations of (orientable) ordinary matroids. OMs capture the basic
properties of sign vectors representing the circuits in a directed
graph or more generally the regions in a central hyperplane arrangement in ${\mathbb R}^d$. OMs obtained from a hyperplane arrangement are called \emph{realizable}. Just as ordinary matroids, oriented matroids may be defined in a multitude of distinct but equivalent ways, see
the book by Bj\"orner et al.~\cite{BjLVStWhZi}. 

Let $U$ be a finite set and let $\covectors$ be a {\it system of sign vectors}, i.e.,
maps from $U$ to $\{\pm 1,0\} = \{-1,0,+1\}$. The elements of $\covectors$ are also referred to as \emph{covectors} and
denoted by capital letters $X, Y, Z$, etc.  
We denote by $\leq$  the product ordering
on $\{ \pm 1,0\}^{U} $ relative to the standard
ordering of signs with $0 \leq -1$ and $0 \leq +1$. 
The \emph{composition} of $X$ and $Y$ is the sign vector $X\circ Y$, where for all $e\in U$ one defines
$(X\circ Y)_e = X_e$ if $X_e\ne 0$  and  $(X\circ Y)_e=Y_e$  if $X_e=0$. The  \emph{topes} of $\covectors$
are the maximal elements of $\covectors$ with respect to $\leq$. A system of sign vectors  $(U,\covectors)$ is called an \emph{oriented matroid} (OM) if  $\covectors$ satisfies the following three axioms:
\begin{itemize}
\item [{\bf (C)}] ({\sf Composition)} $X\circ Y \in  \covectors$  for all $X,Y \in  \covectors$.
\item[{\bf (SE)}] ({\sf Strong elimination})   for each pair $X,Y\in\covectors$ and for each $e\in U$ such that $X_eY_e=-1$,  there exists $Z \in  \covectors$ such that
$Z_e=0$  and  $Z_f=(X\circ Y)_f$  for all $f\in U$ with $X_fY_f\ne -1$.
\item[{\bf (Sym)}] ({\sf Symmetry})  $-\covectors=\{ -X: X\in \covectors\}=\covectors,$ that is, $\covectors$ is closed under sign reversal.
\end{itemize}

Furthermore, a system of sign-vectors $(U,\covectors)$ is \emph{simple} if it has no ``redundant'' elements, i.e., for each $e \in U$,  $\{X_e: X\in \covectors\}=\{+, -,0 \}$ and for each pair $e\neq f$ in $U$,  there exist $X,Y \in \covectors$ with $\{X_eX_f,Y_eY_f\}=\{+, -\}$. 
From (C), (Sym), and (SE) it easily follows that
 the set ${\mathcal T}$ of topes of any simple OM $\covectors$ are  $\{-1,+1\}$-vectors. Therefore  ${\mathcal T}$ can be viewed as a set family (where $-1$ means
 that the  corresponding element does not belong to the set and $+1$ that it belongs). We will only consider simple OMs, without explicitly stating it every time. The \emph{tope graph} of an OM $\covectors$ is the 1-inclusion graph
 $G({\mathcal T})$ of $\mathcal T$ viewed as a set family.   The \emph{Topological Representation Theorem of Oriented Matroids} of \cite{FoLa} characterizes
 tope graphs of OMs as region graphs of
 pseudo-sphere arrangements in a sphere $S^d$~\cite{BjLVStWhZi}. See the bottom part of Figure~\ref{fig:AOM} for an arrangement of pseudo-circles in $S^2$. It is also well-known (see for example ~\cite{BjLVStWhZi}) that tope graphs of OMs are partial cubes and that $\covectors$
 can be recovered from its tope graph  $G({\mathcal T})$ (up to isomorphism). Therefore, we can define all terms in the language of tope graphs. In particular, the isometric dimension of $G({\mathcal T})$ is $|U|$ and its VC-dimension coincides with the dimension $d$ of the sphere $S^d$ hosting a representing pseudo-sphere arrangement.

 Moreover a graph $G$
 is the tope graph of an \emph{affine oriented matroid} (AOM) if $G$ is
 a halfspace of a tope graph of an OM. In particular, tope graphs of AOMs are partial cubes as well.

\subsubsection{COMs: complexes of oriented matroids} Complexes of oriented matroids (COMs) have been introduced and investigated in \cite{BaChKn} as a far-reaching natural common generalization of oriented matroids, affine oriented matroids,
and ample systems of sign-vectors (to be defined below). Some research has been connected to COMs quite quickly, see e.g.~\cite{Baum,Hochstattler,Margolis} and the tope graphs of COMs have been investigated in depth in \cite{KnMa}, see Subsection~\ref{sub:pc-Minors}. 
COMs are defined in a similar way as OMs, simply replacing
the global axiom (Sym) by a weaker local axiom (FS) of face symmetry:  a \emph{complex of oriented matroids (COMs)} is a system of sign vectors  $(U,\covectors)$ satisfying (SE), and the following axiom:
\begin{itemize}
\item[{\bf (FS)}] ({\sf Face symmetry}) $X\circ -Y \in  \covectors$  for all $X,Y \in  \covectors$.
\end{itemize}
As for OMs we generally restrict ourselves to \emph{simple} COMs, i.e., COMs defining simple systems of sign-vectors. 
It is easy to see that (FS) implies (C), yielding that OMs are exactly the COMs containing the zero sign vector
${\bf 0}$, see~\cite{BaChKn}. Also, AOMs are COMs, see~\cite{BaChKn} or~\cite{Baum}. In analogy with realizable OMs, a COM is \emph{realizable} if it is the  systems of sign vectors of the regions in an arrangement $U$ of (oriented) hyperplanes restricted to a convex set of ${\mathbb R}^d$. See Figure~\ref{fig:COMexample} for an example in ${\mathbb R}^2$. For other examples of COMs, see \cite{BaChKn}.

\begin{figure}[htb]
\centering
\includegraphics[width=\textwidth]{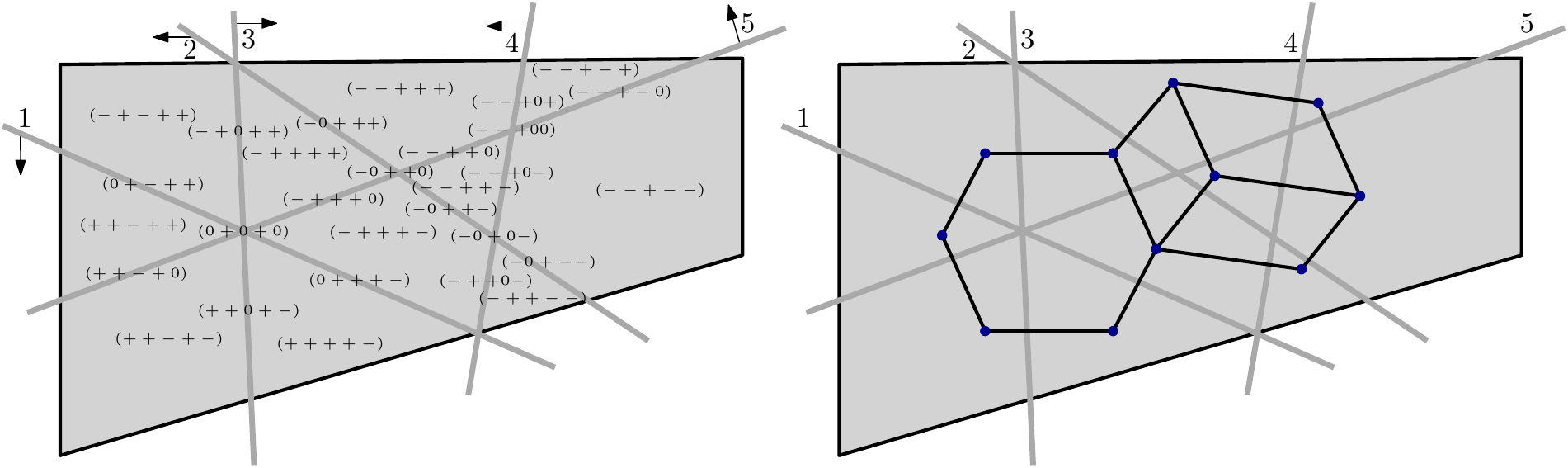}
\caption{The system of sign-vectors associated to an arrangement of hyperplanes restricted to a convex set and the tope graph of the resulting realizable COM.}
\label{fig:COMexample}
\end{figure}

The simple twist between (Sym) and (FS) leads to a rich combinatorial and geometric structure that is build from OM cells but is quite different from OMs. Let $(U,\covectors)$ be a COM and $X$
be a covector of $\covectors$.  The \emph{face} of $X$ is $F(X):=\{ X\circ Y: Y\in \covectors\}$. By \cite[Lemma 4]{BaChKn}, each face $F(X)$ of $\covectors$ is an OM. Moreover, it is shown in \cite[Section 11]{BaChKn}
that replacing each combinatorial face $F(X)$ of  $\covectors$  by a PL-ball, we obtain a contractible cell complex associated to each COM. The \emph{topes} and the \emph{tope graphs} of COMs are defined
in the same way as for OMs. Again, the topes ${\mathcal T}$ are $\{-1,+1\}$-vectors, the tope graph $G({\mathcal T})$ is a partial cubes, and the COM $\covectors$ can be recovered from its tope graph, see~\cite{BaChKn} or~\cite{KnMa}. As for OMs, the isometric dimension of $G({\mathcal T})$ is $|U|$. If a COM is realizable in $\mathbb{R}^d$, then the VC-dimension of $G({\mathcal T})$ is at most $d$.

For each covector $X\in\covectors$, the tope graph of its
face $F(X)$ is a gated subgraph of the tope graph of $\covectors$ \cite{KnMa}: the gate of any tope $Y$ in $F(X)$ is the covector $X\circ Y$ (which is obviously a tope). All this implies that the tope graph of any COM $\covectors$
is obtained by amalgamating gated tope subgraphs of its faces, which are all OMs.

Let $\uparr \covectors:=\{ Y \in \{ \pm 1,0\}^{U}: X \leq Y \text{ for some } X \in  \covectors\}$. Then the \emph{ample systems} (AMPs)\footnote{In the papers on COMs, these systems of sign-vectors are called
lopsided (LOPs).} of sign vectors are those COMs such that $\uparr\covectors=\covectors$~\cite{BaChKn}.
From the definition it follows that any face $F(X)$  consists of  the sign vectors of all faces of the subcube of $[-1,+1]^{U}$ with barycenter  $X$.

\subsubsection{AMPs: ample set families}\label{sub:AMP} Just above we defined ample systems as COMs satisfying $\uparr\covectors=\covectors$.
This is not the first definition of ample systems; all previous definitions define them as families of sets  and not as systems of
sign vectors.
Ample sets  have been introduced by  Lawrence \cite{La}
as asymmetric  counterparts  of  oriented  matroids and have been re-discovered independently by
several works in different contexts~\cite{BaChDrKo,BoRa,Wiedemann}. Consequently,
they received different names: lopsided~\cite{La}, simple~\cite{Wiedemann},
extremal~\cite{BoRa}, and ample~\cite{BaChDrKo,Dr}.  Lawrence~\cite{La} defined ample sets for the
investigation of the possible sign patterns
realized by points of a convex set of $\mathbb{R}^d$.
Ample set families admit a multitude of combinatorial and geometric characterizations
\cite{BaChDrKo,BoRa,La} and comprise many natural
examples arising from discrete geometry, combinatorics, graph theory,
and geometry of groups \cite{BaChDrKo,La} (for applications in machine learning, see \cite{ChChMoWa,MoWa}). 

Let $X$ be a subset of a set $U$ with $m$ elements and let $Q_m=Q(U)$. A {\it $X$-cube}  of $Q_m$ is the 1-inclusion graph
of the set family  $\{ Y\cup X': X'\subseteq X\}$, where $Y$ is a subset of $U\setminus X$. If $|X|=m'$, then any $X$-cube
is a $m'$-dimensional subcube of $Q_m$ and $Q_m$ contains $2^{m-m'}$ $X$-cubes. We call any two $X$-cubes \emph{parallel cubes}.
Recall that $X\subseteq U$ is {\it shattered} by a set family ${\mathcal S}\subseteq 2^{U}$ if $\{ X\cap S: S\in {\mathcal S}\}=2^X$.
Furthermore, $X$ is \emph{strongly shattered} by $\mathcal S$ if the 1-inclusion graph $G({\mathcal S})$ of $\mathcal S$ contains
a $X$-cube.
Denote by $\oX(\SS)$ and $\uX(\SS)$ the families  consisting of all shattered and
of all strongly shattered sets of $\SS$, respectively.  Clearly,
$\uX(\SS)\subseteq \oX(\SS)$ and both $\oX(\SS)$ and $\uX(\SS)$ are closed by
taking subsets, i.e., $\oX(\SS)$ and $\uX(\SS)$ are \emph{abstract simplicial complexes}.
The \emph{VC-dimension} \cite{VaCh} $\vcd(\SS)$ of  $\SS$ is the size of a
largest set shattered by $\SS$, i.e., the dimension of the simplicial
complex $\oX(\SS)$. The fundamental \emph{sandwich lemma} (rediscovered
independently in \cite{AnRoSa,BoRa,Dr,Pa}) asserts
that $\lvert\uX(\SS)\rvert \le \lvert \SS\rvert \le \lvert \oX(\SS)\rvert$.
If $d=\vcd(\SS)$ and $m=\lvert U\rvert$, then $\oX(\SS)$ cannot
contain more than $\Phi_d(m):=\sum_{i=0}^d \binom{m}{i}$ simplices. Thus, the
sandwich lemma yields the well-known \emph{Sauer-Shelah lemma}
\cite{Sauer,Shelah,VaCh} that $\lvert \SS \rvert \le \Phi_d(m)$.

A set family $\SS$ is called \emph{ample} if  $\lvert \SS\rvert=\lvert \oX(\SS)\rvert$ \cite{BoRa,BaChDrKo}.
As shown in those papers this is equivalent to the equality $\oX(\SS)=\uX(\SS)$, i.e., $\SS$ is ample
if and only if any set shattered by $\SS$ is strongly shattered. Consequently, the VC-dimension of an ample family is the
dimension of the largest cube in its 1-inclusion graph. A nice characterization of ample set families was provided
in  \cite{La}: $\SS$ is ample if and only if for any cube $Q$ of $Q_m$ if $Q\cap \SS$ is closed by taking antipodes, then
either $Q\cap \SS=\varnothing$ or $Q$ is included in $G(\SS)$.  The paper \cite{BaChDrKo} provides metric and
recursive characterizations of ample families. For example, it is shown in
\cite{BaChDrKo} that  $\SS$ is ample if and only if any two parallel $X$-cubes of the 1-inclusion graph $G(\SS)$ of $\SS$ can be connected
in $G(\SS)$ by a shortest path of $X$-cubes. This implies that 1-inclusion graphs of ample set families are partial cubes; therefore further
we will speak about \emph{ample partial cubes}. Note that maximum set families (i.e., those which  the Sauer-Shelah lemma is tight) are ample.

\subsubsection{Characterizing tope graphs of  OMs, COMS, and AMPs}\label{sub:pc-Minors} 
In this subsection we recall the characterizations of \cite{KnMa} of tope graphs of COMs, OMs, and AMPs. 
We say that a partial cube $G$ {\it is a COM,} an {\it OM}, an {\it AOM}, or an {\it AMP}
if $G$ is the tope graph of a COM, OM, AOM, or AMP, respectively. Tope graphs of COMs and AMPs are closed under pc-minors and tope graphs of OMs and AOMs are closed under contractions but not under restrictions.
Convex subgraphs of OMs
are COMs and convex subgraphs of tope graphs of uniform OMs are ample. The reverse implications are conjectured in~\cite[Conjecture 1]{BaChKn} and~\cite[Conjecture]{La},
respectively.

As shown in~\cite{KnMa}, a partial cube is the \emph{tope graph of a COM} if and only if all its antipodal subgraphs are gated. Another characterization from the same paper is by an infinite family of excluded pc-minors. This family is denoted by $\mathcal{Q}^-$ and defined as follows. For every $m \geq 4$ there are partial cubes $X_m^1, \ldots, X_m^{m+1}\in \mathcal{Q}^-$. Here, $X_m^{m+1}:=Q_m\setminus\{(0,\ldots,0),(0,\ldots,1,0)\}$, $X_m^{m}=X_m^{m+1}\setminus\{(0,\ldots,0,1)\}$, and $X_m^{m-i}=X_m^{m-i+1}\setminus\{e_{im}\}$. Here $e_{im}$ denotes the vector with all zeroes except positions, $i$ and $m$, where it is one.
See Figure~\ref{fig:COMobstructions} for the members of $\mathcal{Q}^-$ of isometric dimension at most $4$. Note in particular that $X_4^1=SK_4$.
Ample partial cubes can be characterized by the excluding set $\mathcal{Q}^{- -}=\{Q_m^{- -}: m\geq 4\}$, where $Q_m^{- -}=Q_m\setminus\{(0,\ldots,0), (1,\ldots,1)\}$ \cite{KnMa}.
Further characterizations from~\cite{KnMa} yield that OMs are exactly the antipodal COMs, and (as mentioned at the end of Subsection~\ref{sub:OM}) AOMs are exactly the halfspaces of OMs.
On the other hand, ample partial cubes are exactly the
partial cubes in which all antipodal subgraphs are hypercubes.




A central notion in COMs and OMs is the one of the \emph{rank} of $G$, which is the largest $d$ such that $G\in\mathcal{F}(Q_{d+1})$. Hence this notion fits well with the topic of the present paper and combining the families of excluded pc-minors $\mathcal{Q}^{- -}$ and $\mathcal{Q}^{-}$, respectively, with $Q_3$ one gets:

\begin{proposition}\cite[Corollary 7.5]{KnMa}\label{prop:excludedminors}
The class of two-dimensional ample partial cubes coincides with $\mathcal{F}(Q_3, C_6)$. The class of two-dimensional COMs  coincides with $\mathcal{F}(Q_3, SK_4)$.
\end{proposition}
%

%

 \begin{figure}[htb]
\centering
\includegraphics[width=0.80\textwidth]{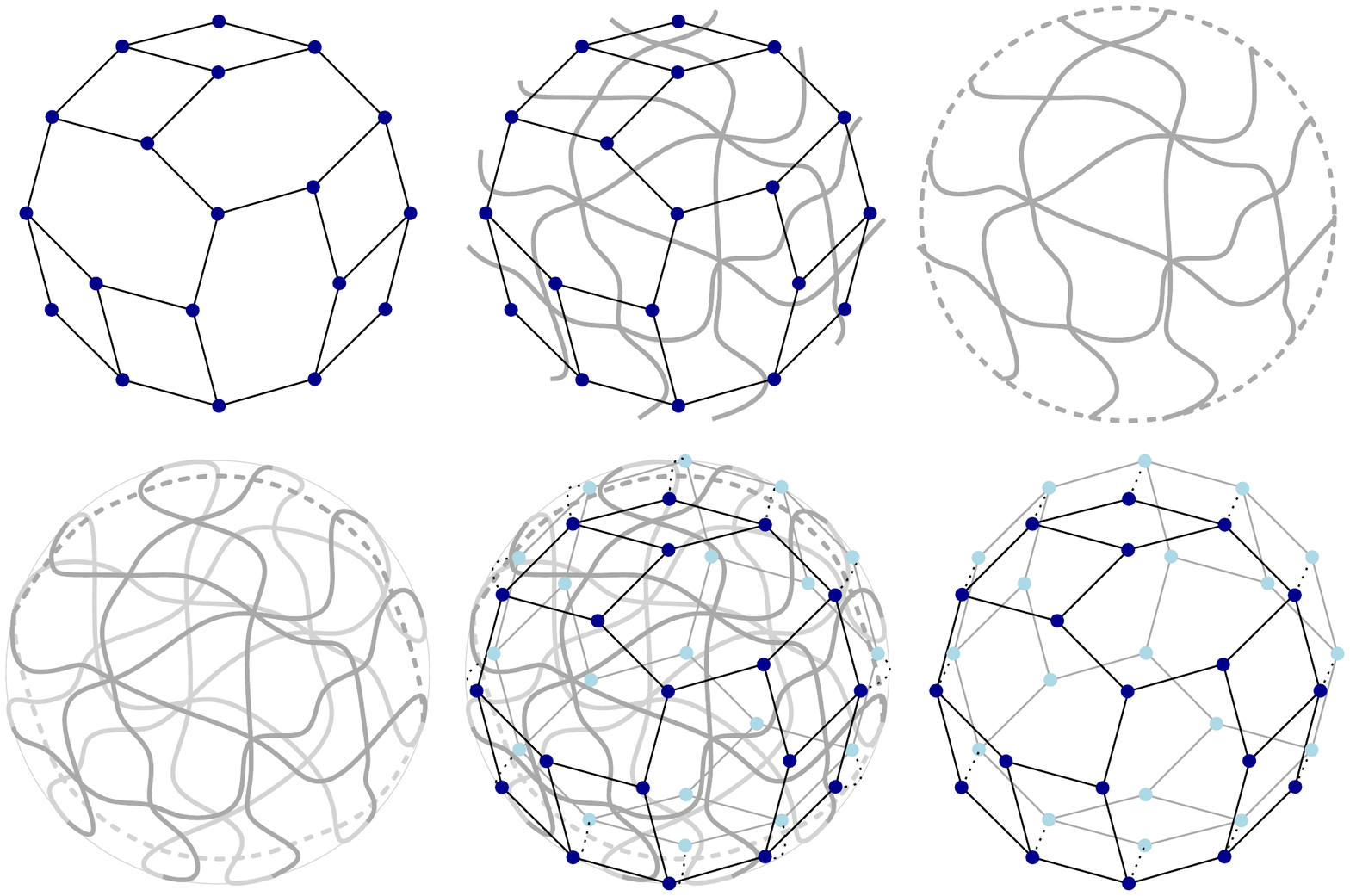}
\caption{From upper left to bottom right: a disk $G$, a pseudoline arrangement $U$ whose region graph is $G$, adding a line $\ell_{\infty}$ to $U$, the pseudocircle arrangement $U'$ obtained from $U\cup\{\ell_{\infty}\}$ with a centrally mirrored copy, the pseudocircle arrangement $U'$ with region graph $G'$, the OM $G'$ with halfspace $G$.}
\label{fig:AOM}
\end{figure}

\subsubsection{Disks}
 A \emph{pseudoline arrangement} $U$ is a family of simple non-closed curves where every pair of curves intersects exactly once and crosses in that point. Moreover, the curves must be extendable to infinity without introducing further crossings. Note that several curves are allowed to cross in the same point. Figure~\ref{fig:AOM} for an example. We say that a partial cube $G$ is a \emph{disk} if it is the region graph of a pseudoline arrangement $U$. The $\Theta$-classes of $G$ correspond to the elements of $U$. Contrary to a convention sometimes made in the literature, we allow a pseudoline arrangement $U$ to be empty, consisting of only one element, or all pseudolines to cross in a single point. These situations yield the simplest examples of disks, namely: $K_1$, $K_2$, and the even cycles. Disks are closed under contraction, since contracting a $\Theta$-class correspond to removing a line from the pseudoline arrangement. It is well-known that disks are tope graphs of AOMs of rank at most $2$. A quick explanation can be found around~\cite[Theorem 6.2.3]{BjLVStWhZi}. The idea is to first add a line $\ell_{\infty}$ at infinity to the pseudoline arrangement $U$ representing $G$. Then embed the disk enclosed by $\ell_{\infty}$ on a hemisphere of $S^2$, such that $\ell_{\infty}$ maps on the equator. Now, mirror the arrangement through the origin of $S^2$ in order to obtain a \emph{pseudocircle arrangement} $U'$. The region graph of $U'$ is an OM $G'$, and the regions on one side of $\ell_{\infty}$ correspond to a halfspace of $G'$ isomorphic to $G$. See Figure~\ref{fig:AOM} for an illustration.

%
%
%
%
%
%


\section{Hyperplanes and isometric expansions}
In this section we characterize the graphs from ${\mathcal F}(Q_{d+1})$ (i.e., partial cubes of VC-dimension $\le d$) via the hyperplanes of their $\Theta$-classes and via the operation of isometric expansion.

\subsection{Hyperplanes}

Let $G$ be isometrically embedded in the hypercube $Q_m$. For a $\Theta$-class $E_i$ of $G$, recall that $G^-_i,G^+_i$ denote the complementary
halfspaces defined by $E_i$ and $\partial G^-_i, G^+_i$ denote their boundaries. The {\it hyperplane} $H_i$ of $E_i$ has the middles of edges of $E_i$ as the vertex-set and two such middles are adjacent in $H_i$ if and only if the corresponding edges
belong to a common square of $G$, i.e.,  $H_i$ is isomorphic to $\partial G^-_i$ and $\partial G^+_i$. 
Combinatorially,  $H_i$ is the 1-inclusion graph of the set family defined by $\partial H^-_i\cup \partial H^+_i$ by removing from each set the element $i$.

\begin{proposition} \label{VCdim_pc} A partial cube $G$ has VC-dimension $\le d$ (i.e., $G$ belongs to ${\mathcal F}(Q_{d+1})$) if and only if each hyperplane $H_i$ of $G$ has VC-dimension $\le d-1$. 
\end{proposition}

\begin{proof} If some hyperplane $H_i$ of $G\in {\mathcal F}(Q_{d+1})$ has VC-dimension $d$, then $\partial G^-_i$ and $\partial G^+_i$ also have VC-dimension $d$ and their union
$\partial H^-_i\cup \partial H^+_i$ has VC-dimension $d+1$. Consequently, $G$ has VC-dimension $\ge d+1$, contrary to Lemma \ref{VCdim_d}.
To prove the converse implication, denote by ${\mathcal H}_{d-1}$ the set of all partial cubes of $G$ in which the hyperplanes have VC-dimension $\le d-1$. We assert that ${\mathcal H}_{d-1}$
is closed under taking pc-minors. First, ${\mathcal H}_{d-1}$ is closed under taking restrictions because the hyperplanes $H'_i$ of any convex subgraph $G'$ of a graph $G\in {\mathcal H}_{d-1}$ are subgraphs of the respective hyperplanes $H_i$ of $G$. Next we show that ${\mathcal H}_{d-1}$ is closed under taking contractions.
Let $G\in {\mathcal H}_{d-1}$ and let $E_i$ and $E_j$ be two different $\Theta$-classes of $G$. Since $\pi_{j}(G)$ is a partial cube, to show that $\pi_{j}(G)$ belongs to ${\mathcal H}_{d-1}$ it suffices to
show that $\partial \pi_{j}(G)^-_i=\pi_{j}(\partial G^-_i)$. Indeed, this would imply that  the $i$th hyperplane of $\pi_{j}(G)$ coincides with the $j$th contraction of the $i$th hyperplane of $G$.
Consequently, this would imply that the VC-dimension of all hyperplanes of $\pi_{j}(G)$ is at most $d-1$.

Pick $v\in \pi_{j}(\partial G^-_i)$. Then $v$ is the image of the edge $v'v''$ of the hypercube $Q_m$ such that at least one of the vertices $v',v''$, say $v'$,  belongs to $\partial G^-_i$. This implies that
the $i$th neighbor $u'$ of $v'$ in $Q_m$ belongs to $\partial G^+_i$. Let $u''$ be the common neighbor of $u'$ and $v''$ in $Q_m$ and $u$ be the image of the edge $u'u''$ by the $j$-contraction. Since $u'\in \partial G^+_i$,
the $i$th edge $vu$ belongs to $\pi_j(G)$, whence  $v\in \partial \pi_{j}(G)^-_i$ and $u\in \partial \pi_{j}(G)^+_i$. This shows  $\pi_{j}(\partial G^-_i)\subseteq \partial \pi_{j}(G)^-_i$.
To prove the converse inclusion, pick a vertex $v\in \partial \pi_{j}(G)^-_i$. This implies that the $i$-neighbor $u$ of $v$ in $Q_m$ belongs to $\partial \pi_{j}(G)^+_i$. As in the previous case, let $v$ be the image
of the $j$-edge $v'v''$ of the hypercube $Q_m$ and let $u'$ and $u''$ be the $i$-neighbors of $v'$ and $v''$ in $Q_m$. Then $u$ is the image of the $j$-edge $u'u''$. Since the vertices $u$ and $v$ belong to $\pi_{j}(G)$, at least one vertex from each of the pairs $\{ u',u''\}$ and $\{ v',v''\}$ belongs to $G$. If one of the edges $u'v'$ or $u''v''$ of $Q_m$ is an edge of $G$, then  $u\in \pi_{j}(\partial G^+_i)$ and $v\in \pi_{j}(\partial G^-_i)$ and we are done.
Finally, suppose that $u'$ and $v''$ are vertices of $G$. Since $G$ is an isometric subgraph of $Q_m$ and $d(u',v'')=2$, a common neighbor $v',u''$ of $u'$ and $v''$ also belongs to $G$ and we fall in the previous case. This shows that $\partial \pi_{j}(G)^-_i\subseteq \pi_{j}(\partial G^-_i)$. Consequently,  ${\mathcal H}_{d-1}$ is closed under taking pc-minors. Since $Q_{d+1}$ does not belong to ${\mathcal H}_{d-1}$, if $G$ belongs to ${\mathcal H}_{d-1}$, then $G$
does not have $Q_{d+1}$ as a pc-minor, i.e., $G\in {\mathcal F}(Q_{d+1})$.
\end{proof}

\begin{corollary} \label{VCdim_1} A partial cube $G$ belongs to ${\mathcal F}(Q_{3})$ if and only if each hyperplane $H_i$ of $G$ has VC-dimension $\le 1$.
\end{corollary}

\begin{remark} In Proposition~\ref{VCdim_pc} it is essential for $G$ to be a partial cube. For example, let ${\mathcal S}$ consist of all subsets of even size of an $m$-element set. Then the 1-inclusion graph $G({\mathcal S})$ of ${\mathcal S}$ consists of isolated vertices (i.e.,  $G({\mathcal S})$  does not contain any edge). Therefore, any hyperplane of $G({\mathcal S})$ is empty, however the VC-dimension of $G({\mathcal S})$ depends on $m$ and can be arbitrarily large.
\end{remark}

By Corollary \ref{VCdim_1}, the hyperplanes of graphs from ${\mathcal F}(Q_{3})$ have VC-dimension 1. However they are not necessarily partial cubes:  any 1-inclusion graph of VC-dimension 1 may occur as a hyperplane of a graph from ${\mathcal F}(Q_{3})$. Thus, it will be useful to establish the metric structure of 1-inclusion graphs of VC-dimension 1.
We  say that a 1-inclusion graph $G$ is a {\it virtual isometric tree} of $Q_m$ if there exists an isometric tree $T$ of $Q_m$ containing $G$ as an induced subgraph. Clearly, each virtually isometric tree is a forest in which each connected component is an isometric subtree of $Q_m$.

\begin{proposition} \label{virtual_isometric_tree} An induced subgraph $G$ of $Q_m$ has VC-dimension $1$ if and only if $G$ is a virtual isometric tree of $Q_m$.
\end{proposition}

\begin{proof} Each isometric tree of $Q_m$ has VC-dimension $1$, thus any virtual isometric tree has VC-dimension $\le 1$. Conversely, let $G$ be an induced subgraph of $Q_m$ of VC-dimension $\le 1$. We will say that two parallelism classes $E_i$ and $E_j$ of $Q_m$ are {\it compatible} on $G$ if one of the four intersections $G^-_i\cap G^-_j, G^-_i\cap G^+_j,G^+_i\cap G^-_j, G^+_i\cap G^+_j$ is empty and {\it incompatible} if the four intersections are nonempty. From the definition of VC-dimension immediately follows that $G$ has VC-dimension 1 if and only if any two parallelism classes of $Q_m$ are compatible on $G$.  By a celebrated result by Buneman \cite{Bu} (see also \cite[Subsection 3.2]{DrHuKoMouSp}), on the vertex set of $G$ one can define a weighted tree $T_0$ with the same vertex-set as $G$ and such that the bipartitions $\{G^-_i,G^+_i\}$ are in bijection with the splits of $T_0$, i.e., bipartitions obtained by removing edges of $T_0$.  The length of each edge of $T_0$ is the number of $\Theta$-classes of $Q_m$ defining the same bipartition of $G$. The distance $d_{T_0}(u,v)$ between two vertices of $T_0$ is equal to the number of parallelism classes of $Q_m$ separating the vertices of $T_0$. We can transform $T_0$ into an isometrically embedded tree $T$ of $Q_m$ in the following way: if the edge $uv$ of $T_0$ has length $k>1$, then replace this edge by any shortest path $P(u,v)$ of $Q_m$ between $u$ and $v$. Then it can be easily seen that $T$ is an isometric tree of $Q_m$, thus $G$ is a virtual isometric tree.
\end{proof}

\subsection{Isometric expansions}\label{subsec:isometric}
In order to characterize median graphs Mulder \cite{Mu} introduced the notion of a convex expansion of a graph. A similar construction of isometric expansion was introduced in \cite{Ch_thesis,Ch_hamming}, with the purpose to characterize isometric subgraphs of hypercubes.  A triplet $(G^1,G^0,G^2)$ is called an {\it isometric cover} of a connected graph $G$, if the following conditions are satisfied:
\begin{itemize}
\item $G^1$ and $G^2$ are two isometric subgraphs of $G$;
\item $V(G)=V(G^1)\cup V(G^2)$ and $E(G)=E(G^1)\cup E(G^2)$;
\item $V(G^1)\cap V(G^2)\ne \varnothing$ and $G^0$ is the
subgraph of $G$ induced by  $V(G^1)\cap V(G^2)$.
\end{itemize}

A graph $G'$ is an {\it isometric expansion} of  $G$ with respect to an isometric cover $(G^1,G^0,G^2)$ of $G$  (notation $G'=\psi(G)$) if $G'$ is obtained from $G$ in the following way:

\begin{itemize}
\item  replace each vertex $x$ of $V(G^1)\setminus V(G^2)$ by a vertex $x_1$ and replace each vertex $x$ of  $V(G^2)\setminus V(G^1)$ by a vertex $x_2$;
\item replace each vertex $x$ of $V(G^1)\cap V(G^2)$ by two vertices $x_1$ and $x_2$;
\item add an edge between two vertices $x_i$ and $y_i,$ $i=1,2$ if and only if  $x$ and $y$ are adjacent vertices of $G^i$, $i=1,2$;
\item add an edge between any two vertices $x_1$ and $x_2$ such that  $x$ is a vertex of $V(G^1)\cap V(G^2)$.
\end{itemize}

\begin{figure}[htb]
\centering
\includegraphics[width=.80\textwidth]{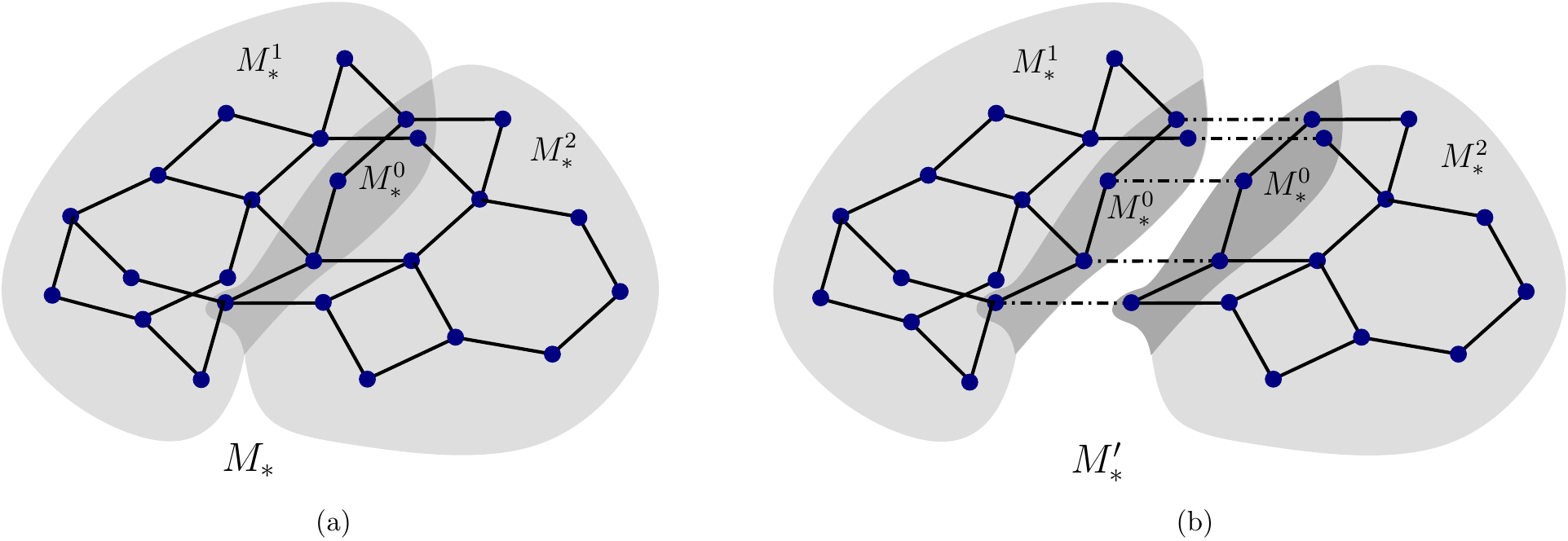}
\caption{(a) The graph $M_*$. (b) The isometric expansion $M'_*$ of $M_*$.}
\label{fig:isom_exp_M_*}
\end{figure}

In other words, $G'$ is obtained by taking a copy of $G^1$, a copy of $G^2$, supposing them disjoint, and adding an edge between any two twins, i.e., two vertices arising from the same vertex of $G^0$.
The following result characterizes all partial cubes by isometric expansions:

\begin{proposition} \label{pc-expansion} \cite{Ch_thesis,Ch_hamming}
A graph is a partial cube if and only if it can be obtained by a sequence of isometric expansions from a single vertex.
\end{proposition}

We also need the following property of isometric expansions:

\begin{lemma}\label{convex_expansion} \cite[Lemma 6]{ChKnMa} If $S$ is a convex subgraph of a partial cube $G$ and $G'$ is obtained from $G$ by an isometric expansion $\psi$, then $S':=\psi(S)$ is a convex subgraph of $G'$.
\end{lemma}

\begin{example}
For partial cubes, the operation of isometric expansion can be viewed as the inverse to the operation of contraction of a $\Theta$-class. For example, the two-dimensional partial cube $M$ can be obtained from
the two-dimensional partial cube $M_*$ (see Figure \ref{fig:M-Theta-class}(b)) via an isometric expansion. In Figure~\ref{fig:isom_exp_M_*} we present another  isometric expansion $M'_*$ of $M_*$. By Proposition \ref{pc-expansion},
$M'_*$  is a partial cube but one can check that it is no longer  two-dimensional.
\end{example}

Therefore, contrary to  all partial cubes, the classes ${\mathcal F}(Q_{d+1})$ are not closed under arbitrary isometric
expansions. In this subsection, we characterize the isometric expansions which preserve the class ${\mathcal F}(Q_{d+1})$.
Let $G$ be isometrically embedded in the hypercube $Q_m=Q(X)$. Suppose that $G$ shatters the subset $Y$ of $X$. For a vertex $v_A$ of $Q(Y)$ (corresponding to a subset $A$ of $Y$), denote by $F(v_A)$ the set of vertices of the hypercube $Q_m$ which projects to $v_A$.
In set-theoretical language, $F(v_A)$ consists of all vertices $v_B$ of $Q(X)$ corresponding to subsets $B$ of $X$ such that $B\cap Y=A$. Therefore, $F(v_A)$ is a subcube of dimension $m-|Y|$ of $Q_m$. Let $G(v_A)=G\cap F(v_A)$.
Since $F(v_A)$ is a convex subgraph of $Q_m$ and $G$ is an isometric subgraph of $Q_m$, $G(v_A)$ is also an isometric subgraph of $Q_m$. Summarizing, we obtain the following property:

\begin{lemma} \label{intersection_with_faces} If $G$ is an isometric subgraph of $Q_m=Q(X)$ which shatters $Y\subseteq X$, then for any vertex $v_A$ of $Q(Y)$, $G(v_A)$ is a nonempty isometric subgraph of $G$.
\end{lemma}

The following lemma establishes an interesting separation property in partial cubes:

\begin{lemma} \label{isometric cover} If $(G^1,G^0,G^2)$ is an isometric cover of an isometric subgraph $G$ of $Q_m=Q(X)$ and $G^1$ and $G^2$ shatter the same subset $Y$  of $X$, then $G^0$ also shatters $Y$.
\end{lemma}

\begin{proof} To prove that $G^0$ shatters $Y$ it suffices to show that for any vertex $v_A$ of $Q(Y)$, $G^0\cap F(v_A)$ is nonempty.  Since $G^1$ and $G^2$ both shatter $Q(Y)$, $G^1\cap F(v_A)$ and $G^2\cap F(v_A)$
are nonempty subgraphs of $G$. Pick any vertices $x\in V(G^1\cap F(v_A))$ and $y\in V(G^2\cap F(v_A))$. Then $x$ and $y$ are vertices of $G(v_A)$. Since by Lemma \ref{intersection_with_faces}, $G(v_A)$ is an isometric subgraph of $Q_m$,
there exists a shortest path $P(x,y)$ of $Q_m$ belonging to $G(v_A)$. Since $(G^1,G^0,G^2)$ is an isometric cover of $G$, $P(x,y)$ contains a vertex $z$ of $G^0$. Consequently, $z\in V(G^0\cap F(v_A))$, and we are done.
\end{proof}

\begin{proposition} \label{expansion-Qd+1} Let $G'$ be obtained from 
$G\in {\mathcal F}(Q_{d+1})$ by an isometric expansion with respect to $(G^1,G^0,G^2)$. Then $G'$ belongs to ${\mathcal F}(Q_{d+1})$ if and only if  $G^0$ has VC-dimension $\le d-1$.
\end{proposition}

\begin{proof} The fact that $G'$ is a partial cube follows from Proposition \ref{pc-expansion}. Let $E_{m+1}$ be the unique $\Theta$-class of $G'$ which does not exist in $G$. Then the halfspaces $(G')^-_{m+1}$ and $(G')^+_{m+1}$ of $G'$ are isomorphic to $G^1$ and $G^2$ and their boundaries $\partial (G')^-_{m+1}$ and $\partial (G')^+_{m+1}$ are isomorphic to $G^0$. If $G'$ belongs to ${\mathcal F}(Q_{d+1})$, by Proposition \ref{VCdim_pc} necessarily $G^0$ has VC-dimension $\le d-1$.

Conversely, let $G^0$ be of VC-dimension $\le d-1$. Suppose 
that $G'$ has VC-dimension $d+1$. Since $G$ has VC-dimension $d$, this implies that any set $Y'$ of size $d+1$ shattered by $G'$ contains the element $m+1$.
Let $Y=Y'\setminus \{ m+1\}$. 
The $(m+1)$th halfspaces $(G')^-_{m+1}$ and $(G')^+_{m+1}$ of $G'$ shatter the set $Y$. Since $(G')^-_{m+1}$ and $(G')^+_{m+1}$ are isomorphic to $G^1$ and $G^2$, the subgraphs $G^1$ and $G^2$ of $G$ both shatter $Y$.
By Lemma \ref{isometric cover}, the subgraph $G^0$ of $G$ also shatters $Y$. Since $|Y|=d$, this contradicts our assumption that $G^0$ has VC-dimension $\le d-1$.
\end{proof}

Let us end this section with a useful lemma with respect to antipodal partial cubes:

\begin{lemma}\label{lem:antipodal}
 If $G$ is a  proper convex subgraph of an antipodal partial cube $H\in\mathcal{F}(Q_{d+1})$, then $G\in\mathcal{F}(Q_{d})$.
\end{lemma}
\begin{proof}
 Suppose by way of contradiction that $G$ has $Q_d$ as a pc-minor. Since convex subgraphs of $H$ are intersections of halfspaces, there exists a $\Theta$-class $E_i$ of $H$ such that $G$ is included in the halfspace $H_i^+$. Since $H$ is antipodal, the subgraph $-G \subseteq H_i^-$ consisting of antipodes of vertices of $G$ is isomorphic to $G$. As $G\subseteq H^+_i$, $-G$ and $G$ are disjoint. Since $G$ has $Q_d$ as a pc-minor,  $-G$ also has $Q_{d}$ as a pc-minor: both those minors are obtained by  contracting the same set $I$ of $\Theta$-classes of $H$; note that $E_i\notin I$. Thus, contracting the $\Theta$-classes from $I$ and all other $\Theta$-classes not crossing the $Q_d$ except $E_i$, we will get an antipodal graph $H'$, since antipodality is preserved by contractions. Now, $H'$ consists of two copies of $Q_d$ separated by $E_i$. Take any vertex $v$ in $H'$. Then there is a path from $v$ to $-v$ first crossing all $\Theta$-classes of the cube containing $v$ and then $E_i$, to finally reach $-v$. Thus, $-v$ is adjacent to $E_i$ and hence every vertex of $H'$ is adjacent to $E_i$.
 Thus $H'=Q_{d+1}$, contrary to the assumption that $H\in\mathcal{F}(Q_{d+1})$.
\end{proof}


\section{Gated hulls of 6-cycles}  In this section, we prove that in two-dimensional partial cubes the gated hull of any 6-cycle $C$ is either $C$, or  $Q^-_3$, or a
maximal full subdivision of $K_n$. 


\subsection{Full subdivisions of $K_n$}
A {\it full subdivision of $K_n$} (or {\it full subdivision} for short) is the graph $SK_n$ obtained from the complete graph $K_n$ on $n$ vertices by subdividing
each edge of $K_n$ once; $SK_n$ has $n+\binom{n}{2}$ vertices and $n(n-1)$ edges. The $n$ vertices of $K_n$ are called {\it original} vertices of $SK_n$ and
the new vertices are called {\it subdivision} vertices. Note that $SK_3$ is the 6-cycle $C_6$. 
Each $SK_n$ can be isometrically embedded into the $n$-cube $Q_n$ in such a way that each original vertex $u_i$ is encoded by the one-element set $\{ i\}$  and each vertex $u_{i,j}$ subdividing the edge $ij$ of $K_n$
is encoded by the 2-element set $\{ i,j\}$ (we call this embedding of $SK_n$ a {\it standard embedding}). If we add to $SK_n$ the vertex $v_{\varnothing}$ of $Q_n$ which corresponds to the empty set $\varnothing$,
we will obtain the partial cube $SK^*_n$. Since both $SK_n$ and $SK^*_n$
are encoded by subsets of size $\le 2$, those graphs have VC-dimension 2. Consequently, we obtain:

\begin{lemma} \label{SKn} For any $n$, $SK_n$ and $SK^*_n$ are two-dimensional partial cubes.
\end{lemma}

\begin{figure}[htb]
\centering
\includegraphics[width=.70\textwidth]{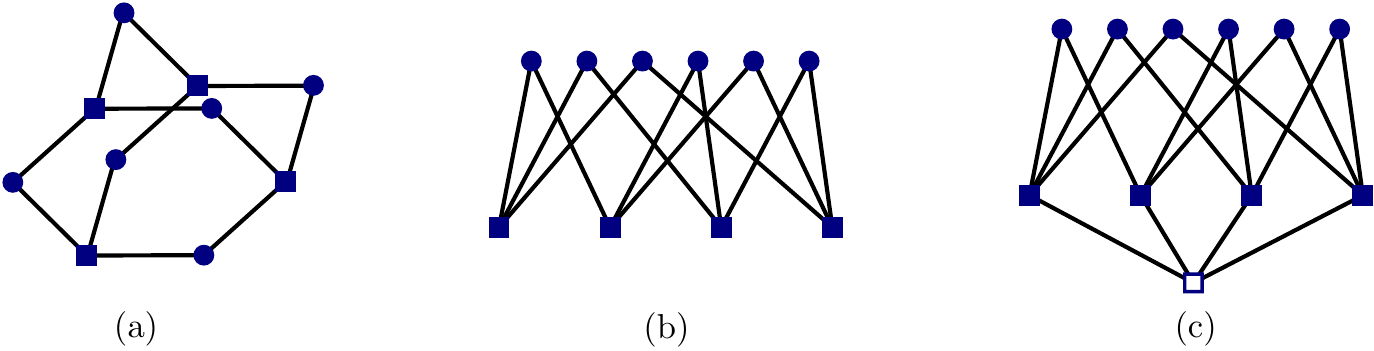}
\caption{(a) An isometric embedding of $SK_4$ into $Q_4$. (b) A standard embedding of $SK_4$. (c) A completion of $SK_4$ to $SK^*_4$.}
\label{fig:representation_SK_4}
\end{figure}

\begin{example} Our running example $M$ contains two isometrically embedded copies of $SK_4$. In Figure~\ref{fig:representation_SK_4}(a)\&(b)
we present two isometric embeddings of $SK_4$ into the 4-cube $Q_4$, the second one is the standard embedding of $SK_4$.  The
original and subdivision vertices are illustrated by squares and circles, respectively. Figure~\ref{fig:representation_SK_4}(c)
describes the completion of $SK_4$ to $SK^*_4$.
\end{example}

\begin{lemma} \label{standardSKn} If $H=SK_n$ with $n \ge 4$ is an isometric subgraph of a partial cube $G$, then $G$ admits an isometric embedding  into a hypercube such
that the embedding of $H$ is standard.
\end{lemma}

\begin{proof} Pick any original vertex of $H$ as the base point $b$ of $G$ and consider the standard isometric embedding $\varphi$ of $G$ into $Q_m$. Then $\varphi(b)=\varnothing$.
In $H$ the vertex $b$ is adjacent to $n-1\ge 3$  subdivision vertices of $H$.
Then for each of those vertices $v_i, i=1,\ldots,n-1,$ we can suppose that $\varphi(v_i)=\{ i\}$. Each $v_i$ is adjacent in $H$ to an original vertex $u_i\ne b$. Since $H$
contains at least three such original vertices and they have pairwise distance 2, one can easily check that the label $\varphi(u_i)$ consists of $i$ and an element common to all such vertices, denote it by $n$.
Finally, the label of any subdivision vertex $u_{i,j}$ adjacent to the original vertices $u_i$ and $u_j$ is $\{ i,j\}$.
Now consider an isometric embedding $\varphi'$ of $G$ defined by setting $\varphi'(v)=\varphi(v)\Delta \{ n\}$ for any vertex $v$ of $G$. Then $\varphi'$ provides a standard embedding
of $H$: $\varphi'(b)=\{ n\},$ $\varphi'(u_i)=\{ i\}$ for any original vertex $u_i$, and $\varphi'(v_i)=\{ i,n\}$ for any subdivision vertex $v_i$ adjacent to $b$ and $\varphi'(u_{i,j})=\{ i,j\}$
for any other subdivision vertex $u_{i,j}$.
\end{proof}

By Lemma \ref{standardSKn}, when a full subdivision $H=SK_n$ of a graph $G\in {\mathcal F}(Q_3)$ is fixed, we  assume that $G$ is isometrically embedded in a hypercube so that $H$ is standardly embedded.

We describe next the isometric expansions of $SK_n$ which result in two-dimensional partial cubes. An isometric expansion of a partial cube $G$ with respect to $(G^1,G^0,G^2)$ is called {\it peripheral} if at least one of the subgraphs $G^1,G^2$ coincides with $G^0$, i.e., $G^1\subseteq G^2$ or $G^2\subseteq G^1$.

\begin{lemma} \label{expansionSKn} If $G'$ is obtained from $G:=SK_n$ with $n \ge 4$ by an isometric expansion with respect to $(G^1,G^0,G^2)$, then $G'\in {\mathcal F}(Q_3)$ if and only if this is a peripheral expansion and $G^0$
is an isometric tree of $SK_n$.
\end{lemma}

\begin{proof} The fact that an isometric expansion  of $SK_n$, such that $G^0$ is an isometric tree, belongs to ${\mathcal F}(Q_3)$ follows from Proposition \ref{expansion-Qd+1} and Lemma \ref{SKn}. Conversely,  suppose that
$G'$ belongs to ${\mathcal F}(Q_3)$. By Proposition \ref{expansion-Qd+1}, $G^0$  has VC-dimension $\le 1$ and by Proposition \ref{virtual_isometric_tree} $G^0$ is a virtual tree. It suffices to prove that $G^1$ or $G^2$
coincides with $G^0$. Indeed, since $G^1$ and $G^2$ are isometric subgraphs of $SK_n$, this will also imply that $G^0$ is an isometric tree.
We distinguish two cases.

\medskip\noindent
{\bf Case 1.}
First, let $G^0$ contain two original vertices $u_{i}$ and $u_{j}$.
Since $u_{i}$ and $u_{j}$ belong to $G^1$ and $G^2$ and those two subgraphs are isometric subgraphs of $G$, the unique common neighbor $u_{i,j}$ of $u_{i}$ and $u_{j}$
must belong to $G^1$ and $G^2$, and thus to $G^0$. If another original vertex $u_{k}$ belongs to $G^0$, then  the four vertices $u_{i,j}, u_{i}, u_{j},u_{k}$ of $G^0$
shatter the set $\{ i,j\}$, contrary to the assumption that $G^0$  has VC-dimension $\le 1$ (Proposition \ref{expansion-Qd+1}). This implies that each other original
vertex  $u_{k}$ either belongs to $G^1\setminus G^2$ or to $G^2\setminus G^1$.
If there exist original vertices $u_{k}$ and $u_{\ell}$ such that $u_{k}$ belongs to $G^1\setminus G^2$ and $u_{\ell}$ belongs to $G^2\setminus G^1$, then their unique common neighbor $u_{k,\ell}$ necessarily belongs to $G^0$. But in this case the
four vertices $u_{i,j}, u_{i}, u_{j},u_{k,\ell}$ of $G^0$ shatter the set $\{ i,j\}$. Thus we can suppose that all other original vertices $u_{k}$ belong to $G^1\setminus G^2$.
Moreover, for the same reason and since $G^1$ is an isometric subgraph of $G$, any vertex $u_{k,\ell}$ with $\{ k,\ell\}\ne \{ i,j\}$ also belongs to $G^1\setminus G^2$.
Since $G^1$ is an isometric subgraph of $G$, for any $k\ne i,j$, the vertices $u_{i,k}, u_{j,k}$ belong to $G^1$. Therefore $G^1=G$ and $G^0=G^2$. Since $G^2$ is an isometric subgraph of $G$ and $G^0$  has
VC-dimension $\le 1$, $G^0$ is an isometric subtree of $G$.

\medskip\noindent
{\bf Case 2.}
Now, suppose that $G^0$ contains at most one original vertex. Let $A^1$ be the set of original vertices belonging to $G^1\setminus G^2$ and $A^2$ be the set of original vertices
belonging to $G^2\setminus G^1$. First suppose that $|A^1|\ge 2$ and $|A^2|\ge 2$, say
$u_{1},u_{2}\in A^1$ and $u_{3},u_{4}\in A^2$. But then the vertices $u_{1,3},u_{1,4},u_{2,3},u_{2,4}$ must belong to $G^0$. Since those four vertices shatter
the set $\{ 1,3\}$, we obtain a contradiction that $G^0$ has VC-dimension $\le 1$. Hence, one of the sets $A^1$ or $A^2$ contains at most one vertex. Suppose without loss of generality that $A^1$ contains
at least $n-2$ original vertices $u_{1},u_{2},\ldots, u_{n-2}$. First suppose that $G^1$ contains all original vertices. Then since $G^1$ is an isometric subgraph of $G$, each subdivision vertex $u_{i,j}$ also
belongs to $G^1$. This implies that $G^1=G$ and we are done. Thus suppose that the vertex $u_{n}$ does not belong to $A^1$. Since $G^0$ contains at most one original vertex, one of the
vertices $u_{n-1},u_{n}$, say  $u_{n}$,
must belong to $A^2$ (i.e., to $G^2\setminus G^1$). This implies that all vertices $u_{i,n}, i=1,\ldots,n-2$ belong to $G^0$. Since $n\ge 4$ and $u_{n}$ is the unique common neighbor of the vertices $u_{i,n}$ and
$u_{j,n}$ with $i\ne j$ and $1\le i,j\le n-2$ and $G^1$ is an isometric subgraph of $G$, necessarily $u_{n}$ must be a vertex of $G^1$, contrary to our assumption that $u_{n}\in A^2$. This contradiction concludes
the proof of the lemma.
\end{proof}

  \begin{corollary}\label{SKnCopie}
    If $G \in \mathcal{F}(Q_3)$ and  $G$ contains $SK_n$ with $n \ge 4$ as a pc-minor, then $G$ contains $SK_n$ as a convex subgraph.
  \end{corollary}

\begin{proof} Suppose by way of contradiction that $G'$ is a smallest graph in $\mathcal{F}(Q_3)$ which contains $SK_n$ as a pc-minor but does not contain $SK_n$ as a convex subgraph. This means that any contraction
of $G'$ along a $\Theta$-class of $G'$ that do not cross the $SK_n$ pc-minor, also contains this $SK_n$ as a pc-minor. We denote the resulting graph by $G$. Since $G\in \mathcal{F}(Q_3)$,
by minimality choice of $G'$, $G$ contains $SK_n$ as a convex subgraph, denote this subgraph
by $H$. Now, $G'$ is obtained from $G$ by an isometric expansion. By Lemma \ref{convex_expansion}, $H'=\psi(H)$ is a convex subgraph of $G'$. Since $G'\in  \mathcal{F}(Q_3)$, by Lemma \ref{expansionSKn} this isometric
expansion restricted to $H=SK_n$ is  a peripheral expansion. This implies that the image of $H$
under this expansion is a convex subgraph $H'$ of $G'$ which contains a copy of $SK_n$ as a convex subgraph, and thus $G'$ contains a convex copy of $SK_n$.
\end{proof}

\begin{lemma} \label{6cycle} If $C=SK_3$ is an isometric 6-cycle of $G\in {\mathcal F}(Q_3)$, then  $C$ is convex or its convex hull is $Q^-_3$.
\end{lemma}

\begin{proof} The convex hull of $C$ in $Q_m$ is a 3-cube $Q$ and $\conv(C)=Q\cap V(G)$. Since $G$ belongs to ${\mathcal F}(Q_3)$, $Q$ cannot be included in  $G$.
Hence either $\conv(C)=C$ or $\conv(C)=Q^-_3$.
\end{proof}

\subsection{Gatedness of full subdivisions of $K_n$}
The goal of this subsection is to prove the following result:

\begin{proposition} \label{gatedSKn} If $H=SK_n$ with $n \ge 4$ is a convex subgraph of $G\in {\mathcal F}(Q_3)$ and $H$ is not included in a larger full subdivision of $G$, then $H$ is a gated subgraph of $G$.
\end{proposition}

\begin{proof} The proof of Proposition \ref{gatedSKn} uses the results of previous subsection and two claims.

\begin{claim} \label{convexSKn} If $H=SK_n$ with $n \ge 4$ is an isometric subgraph of $G\in {\mathcal F}(Q_3)$, then either $H$ extends in $G$ to $SK^*_n$ or $H$ is a convex subgraph of $G$.
\end{claim}

\begin{proof} Suppose by way of contradiction that $H=SK_n$ does not extend in $G$ to $SK^*_n$ however $H$ is not convex. Then there exists a vertex $v\in V(G)\setminus V(H)$ such that $v\in I(x,y)$ for
two vertices $x,y\in V(H)$. First note that $x$ and $y$ cannot be both original vertices. Indeed, if $x=u_i$ and $y=u_j$, then in $Q_m$ the vertices $x$ and $y$ have two common neighbors:
the subdivision vertex $u_{i,j}$ and $v_{\varnothing}$. But $v_{\varnothing}$ is adjacent in $Q_m$ to all original vertices of $H$,
thus it cannot belong to $G$ because $H=SK_n$ does not extend to $SK^*_n$. Thus, further we can suppose that the vertex $x$ is a subdivision vertex, say $x=u_{i,j}$. We distinguish several
cases depending of the value of $d(x,y)$.

\medskip\noindent
{\bf Case 1.} $d(x,y)=2$.

This implies that $y=u_{i,k}$  is also a subdivision vertex and $x$ and $y$ belong in $H$ to a common isometric 6-cycle $C$. Since $v$ belongs to $\conv(C)$, Lemma \ref{6cycle} implies that $v$ is adjacent to the third subdivision vertex $z=u_{j,k}$ of $C$. Hence $v=\{ i,j,k\}$. Since $n\ge 4$, there exists $\ell\ne i,j,k$ such that $\{ \ell\}$ is an original vertex of $H$ and $\{ i,\ell\},\{ j,\ell\},$ and $\{ k,\ell\}$ are subdivision vertices of $H$. Contracting $\ell$, we will obtain a forbidden $Q_3$.

\medskip\noindent
{\bf Case 2.} $d(x,y)=3$.

This implies that $y=u_k$ is an original vertex with $k\ne i,j$. Then again  the vertices $x$ and $y$ belong in $H$ to a common isometric 6-cycle $C$. Since $v$ belongs to $\conv(C)$, Lemma \ref{6cycle} implies that either $v$ is adjacent to $u_i,u_j$, and $u_k$ or to $u_{i,j},u_{i,k}$, and $u_{j,k}$, which was covered by the Case 1.

\medskip\noindent
{\bf Case 3.} $d(x,y)=4$.

This implies that $y=u_{ k,\ell}$ is a subdivision vertex with $k,\ell\ne i,j$. In view of the previous cases, we can suppose that $v$ is adjacent to $x$ or to $y$, say $v$ is adjacent to $x$. Let $Q$ be the convex hull of $\{ x,y\}$ in $Q_m$. Then $Q$ is a 4-cube and $x=\{ i,j\}$ has 4 neighbors in $Q$: $\{ i\}, \{ j\}, \{ i,j,k\}$ and $\{ i,j,\ell\}$. The vertices $\{ i\}, \{ j\}$ are original vertices of $H$. Thus suppose that $v$ is one of the vertices $\{ i,j,k\},\{ i,j,\ell\}$, say $v=\{ i,j,k\}$.
But then $v$ is adjacent to $\{ j,k\}$, which is a subdivision vertex of $H$ and we are in the conditions of Case 1. Hence $H$ is a convex subgraph of $G$.
\end{proof}

\begin{claim} \label{claimSKn} If $H=SK_n$ with $n \ge 4$ is a convex subgraph of $G\in {\mathcal F}(Q_3)$ and $H$ is not included in a larger full subdivision in $G$, then the vertex
$v_{\varnothing}$ of $Q_m$ is adjacent only to the original vertices $u_1,\ldots,u_n$ of $H$.
\end{claim}

\begin{proof}
Since $H$ is convex, the vertex $v_{\varnothing}$ of $Q_m$  is not a vertex of $G$. Let $u_i=\{ i\}, i=1,\ldots,n$ be the original vertices of $H$. Suppose that in $Q_m$ the
vertex $v_{\varnothing}$ is adjacent to a vertex $u$ of $G$, which is not included in $H$, say $u=\{ n+1\}$. Since $u$ and each $u_i$ has in $Q_m$ two common neighbors $v_{\varnothing}$
and $u_{i,n+1}=\{ i,n+1\}$ and since $G$ is an isometric subgraph of $Q_m$, necessarily each vertex $u_{i,n+1}$ is a vertex of $G$. Consequently, the vertices of $H$ together with the vertices $u,u_{1,n+1},\ldots,u_{n,n+1}$
define an isometric subgraph $H'=SK_{n+1}$ of $Q_m$. Since $v_{\varnothing}$ does not belong to $G$, by Claim \ref{convexSKn} $H'$ is convex, contrary to the assumption that $H$ is not included in a larger convex
full subdivision of $G$. Consequently, the neighbors in $G$ of $v_{\varnothing}$ are only the original vertices $u_1,\ldots,u_n$ of $H$.
\end{proof}
%
%
Now, we prove Proposition \ref{gatedSKn}. Let $G\in {\mathcal F}(Q_3)$ be an isometric subgraph of the cube $Q_m$ in such that the embedding of $H$ is standard.
Let $Q$ be the convex hull of $H$ in $Q_m$; $Q$ is a cube of dimension $n$
and a gated subgraph of $Q_m$.  Let $v$ be a vertex of $G$ and $v_0$ be the gate of $v$ in $Q$. To prove that $H$ is gated it suffices to show that $v_0$ is a vertex of $H$. Suppose by way of contradiction
that $H$ is not gated in $G$ and among  the vertices of $G$ without a gate in $H$ pick a vertex $v$ minimizing the distance $d(v,v_0)$. Suppose that $v$ is encoded by the set $A$. Then
its gate $v_0$ in $Q_m$ is encoded by the set $A_0:=A\cap \{ 1,\ldots,n\}$. If $|A_0|=1,2$, then $A_0$ encodes an original or subdivided vertex of $H$, therefore $v_0$ would belong to $H$,
contrary to the choice of $v$. So, $A_0=\varnothing$ or $|A_0|>2$.

First suppose that $A_0=\varnothing$, i.e., $v_0=v_{\varnothing}$. Since $v_\varnothing$ is adjacent only to the original vertices of $H$, by Claim \ref{claimSKn} all original vertices of $H$ have distance $k=d(v,v_{\varnothing})+1\ge 3$ to $v$.
From the choice of $v$ it follows that $I(v,u_i)\cap I(v,u_j)=\{ v\}$ for any two original vertices $u_i$ and $u_j$, $i\ne j$.
Indeed, if $I(v,u_i)\cap I(v,u_j)\ne \{ v\}$ and $w$ is a neighbor of $v$ in $I(v,u_i)\cap I(v,u_j)$, then $d(w,u_i)=d(w,u_j)=k-1$. Therefore the gate $w_0$ of $w$  in $Q$ has distance at most $k-2$ from $w$,
yielding that $d(v,w_0)=k-1$. This is possible only if $w_0=v_0$. Therefore, replacing $v$ by $w$ we will get a vertex of $G$ whose gate $w_0=v_0$ in $Q$ does not belong to $H$ and for which $d(w,w_0)<d(v,v_0)$,
contrary to the minimality in the choice of $v$.  Thus $I(v,u_i)\cap I(v,u_j)=\{ v\}$.
Let $A=\{ n+1,\ldots, n+k-1\}$.

If $k=3$, then 
$v$ is encoded by  $A=\{ n+1,n+2\}$. By Claim \ref{claimSKn}, any shortest path of $G$ from $u_i=\{ i\}$ to $v$ must be of the form $(\{ i\}, \{ i,\ell\}, \{ \ell\}, \{ n+1,n+2\})$, where $\ell\in \{ n+1,n+2\}$. Since we have at least four original vertices, at least two of such shortest paths of $G$ will pass via the same neighbor $\{ n+1\}$ or $\{ n+2\}$ of $v$, contrary to the assumption that $I(v,u_i)\cap I(v,u_j)=\{ v\}$ for any $u_i$ and $u_j$, $i\ne j$.
If $k\ge 4$,  let $G'=\pi_{n+1}(G)$ and $H'=\pi_{n+1}(H)$ be the images of $G$ and $H$ by contracting the edges of $Q_m$ corresponding to the coordinate $n+1$. Then $G'$ is an isometric subgraph of the hypercube $Q_{m-1}$ and $H'$ is a full subdivision isomorphic to $SK_n$ isometrically embedded in $G'$. Let also $v',v'_{\varnothing,}$ and $u'_i, i=1,\ldots,n,$ denote the images of the vertices $v, v_{\varnothing},$ and $u_i$ of $G$. Then $u'_1,\ldots, u'_n$ are the original vertices of $H'$. Notice also
that $v'$ has distance $k-1$ to all original vertices of $H'$ and distance $k-2$ to $v'_{\varnothing}$. Thus in $G'$ the vertex $v'$ does not have a gate in $H'$. By the minimality in the choice of $v$ and $H$, either $H'$ is not convex in
$G'$ or $H'$ is included in a larger full subdivision of $G'$. If $H'$ is not convex in $G'$, by Claim \ref{convexSKn} $v'_{\varnothing}$ must be a vertex of $G'$. Since $v_{\varnothing}$ is not a vertex of $G$,
this is possible only if the set $\{ n+1\}$ corresponds to a vertex of $G$. But we showed in Claim \ref{claimSKn}  that the only neighbors of $v_{\varnothing}$ in $G$ are the original vertices of $H$. This contradiction shows that $H'$ is a convex. Therefore, suppose that $H'$ is included in a larger full subdivision $H''=SK_{n+1}$ of $G'$. Denote by $u'_{\ell}=\{ \ell\}$ the original vertex of $H''$ different from the vertices $u'_i, i=1,\ldots,n$; hence $\ell\notin \{ 1,\ldots,n\}$.
Since $u'_{\ell}$ is a vertex of $G'$ and in $Q_m$ the set $\{ \ell\}$ does not correspond to a vertex of $G$, necessarily the set $\{ n+1,\ell\}$ is a vertex of $G$ in $Q_m$. Therefore, we are in the conditions of the previous subcase, which was shown to be impossible. This concludes the analysis of case  $A_0=\varnothing$.


Now, suppose that $|A_0|\ge 3$ and let $A_0=\{ 1,2,3,\ldots,k\}$. This implies that the vertices $u_1,u_2,u_3$ are original vertices  and $u_{1,2},u_{1,3},u_{2,3}$ are subdivision vertices of $H$. Since $H=SK_n$ with $n\ge 4$, $H$ contains an original vertex $u_{\ell}$ with $\ell\ge 4$, say $\ell=4$. But then the sets corresponding to the vertices $u_1,u_2,u_3,u_4,u_{1,2},u_{1,3},u_{2,3},$ and $v$ of $G$  shatter the set $\{ 1,2,3\}$, contrary to the assumption that $G\in {\mathcal F}(Q_3)$. This concludes the case $|A_0|\ge 3$.
Consequently, for any vertex $v$ of $G$ the gate $v_0$ of $v$ in $Q$ belongs to $H$. This shows that $H$ is a gated subgraph of $G$ and concludes the proof of the proposition.
\end{proof}

\subsection{Gated hulls of 6-cycles}\label{gh-6cycle}  The goal of this subsection is to prove the following result:

\begin{proposition} \label{gatedhullC6} If $C$ is an induced (and thus isometric) 6-cycle of $G\in {\mathcal F}(Q_3)$, then the gated hull $\gate(C)$ of $C$ is either $C$, or $Q_3^-$, or a full subdivision.
\end{proposition}

\begin{proof} If $C$ is included in a maximal full subdivision $H=SK_n$ with $n\ge 4$, by Proposition \ref{gatedSKn} $H$ is gated. Moreover,
one can directly check that any vertex of $H\setminus C$ must be included in the gated hull of $C$, whence $\gate( C )=H$. Now suppose that $C$ is not included in any  full subdivision $SK_n$ with $n\ge 4$.
By Lemma \ref{6cycle}, $S:=\conv(C)$ is either $C$ or $Q^-_3$. In this case we assert that $S$ is gated and thus $\gate(C)=\conv(C)$.  Suppose that $G$ is a two-dimensional partial cube of smallest size
for which this is not true. Let $v$ be a vertex of $G$ that has no gate in $S$ and is as close as possible to $S$, where $d_G(v,S)=\min \{ d_G(v,z): z\in S\}$
is the distance from $v$ to $S$. Given a $\Theta$-class $E_i$ of $G$, let $G':=\pi_i(G)$, $S':=\pi_i(S)$, and $C'=\pi_i(C)$. For a vertex $u$ of $G$, let $u':=\pi_i(u)$.

Since any convex subgraph of $G$ is the intersection of halfspaces, if all $\Theta$-classes of $G$ cross $S$, then $S$ coincides with $G$, contrary to the  choice of $G$. Thus $G$ contains $\Theta$-classes not crossing $S$.
First suppose that there exists a $\Theta$-class $E_i$ of $G$ not crossing $S$ such that $S'$ is convex in $G'$. Since $G'\in {\mathcal F}(Q_3)$, by Lemma \ref{6cycle} either the 6-cycle $C'$ is
convex or its convex hull in $G'$ is $Q^-_3$. Since the distance in $G'$ between $v'$ and any vertex of $S'$ is either the same as the distance in $G$ between $v$ and the corresponding vertex of $S$
(if $E_i$ does not separate $v$ from $S$) or is one less than the corresponding distance in $G$ (if $v$ and $S$ belong to complementary halfspaces defined by $E_i$), $S'$ is not gated in $G'$, namely the vertex $v'$ has no gate in $S'$.
Therefore, if $S'=Q^-_3$, then contracting all $\Theta$-classes of $G'$ separating $S'$ from $v'$, we will get $Q_3$ as a pc-minor, contrary to the assumption that $G$ and $G'$ belong to $\mathcal{F}(Q_3)$. This implies that
$S'=C'$ and thus that $S=C$. Moreover, by minimality of $G$, the 6-cycle $C'$ is included in a maximal full subdivision $H'=SK_n$ of $G'$.
By Proposition \ref{gatedSKn}, $H'$ is a gated subgraph of $G'$. Let $w'$
 be the gate of $v'$ in $H'$ (it may happen that $w'=v'$).  Since $C'$ is not gated, necessarily $w'$ is not a vertex of $C'$. For the same reason, $w'$ is not adjacent to a vertex of $C'$.
 The graph $G$ is obtained from $G'$ by an isometric expansion $\psi_i$ (inverse to $\pi_i$).
 By Lemma \ref{expansionSKn},  $\psi_i$, restricted to $H'$, is a peripheral expansion along an isometric tree of $H'$.
By Corollary \ref{SKnCopie}, $G$ contains an isometric subgraph isomorphic to $H'$. By the choice of $E_i$, $C$ does not cross $E_i$, and this implies that in $G$ the convex cycle $C$
is contained in a full subdivision of $K_n$, contrary to the choice of $C$.

 Now, suppose that for any $\Theta$-class $E_i$ of $G$ not crossing $S$, $S'$ is not convex in $G'$. Since $C'$ is an isometric 6-cycle of $G'$, $G'\in {\mathcal F} (Q_3)$, and the 6-cycle $C'$
 is not convex in $G'$, by Lemma \ref{6cycle} we conclude that the convex hull of $C'$ in $G'$ is $Q^-_3$ and this $Q^-_3$ is different from $S'$.
 Hence $S'=C'$ and $S=C$.  This implies that there exists a vertex $z'$ of $G'$ adjacent to three vertices $z'_1,z'_2,$ and $z'_3$ of $C'$. Let $z_1,z_2,z_3$ be the three preimages in $C$ of the vertices $z'_1,z'_2,z'_3$.
 Let also $y,z$ be the preimages in the hypercube $Q_m$ of the vertex $z'$. Suppose that $y$ is adjacent to $z_1,z_2,z_3$ in $Q_m$. Since $C'$ is the image of the convex 6-cycle of $G$, this implies
 that $y$ is not a vertex of $G$ while $z$ is a vertex of $G$. Since $G$ is an isometric subgraph of $Q_m$, $G$ contains a vertex $w_1$ adjacent to $z$ and $z_1$, a vertex $w_2$ adjacent to $z$
 and $z_2$, and a vertex $w_3$ adjacent to $z$ and $z_3$. Consequently, the vertices of $C$ together with the vertices $z,w_1,w_2,w_3$ define a full subdivision $SK_4$, contrary to our assumption
 that $C$ is  not included in such a subdivision. This shows that the convex hull of the 6-cycle $C$ is gated.
\end{proof}

\section{Convex and gated hulls of long isometric cycles}

In the previous section we described the structure of gated hulls of 6-cycles in two-dimensional partial cubes. In this section, we provide a description of convex and gated hulls of {\it long isometric cycles}, i.e., of isometric cycles of length $\ge 8$. We prove that convex hulls of long isometric cycles are disks, i.e., the region graphs of pseudoline arrangements.  Then  we show that all such disks are gated. 
In particular, this implies that convex long cycles in two-dimensional partial cubes are gated.

\subsection{Convex hulls of long isometric cycles}

A two-dimensional  partial cube $D$ is called a \emph{pseudo-disk} if $D$ contains an isometric cycle $C$ such that  $\conv(C)=D$;
$C$ is called the {\it boundary} of $D$ and is denoted by $\partial D$. If $D$ is the convex hull of an isometric cycle $C$ of $G$,
then we say that $D$ is a pseudo-disk of $G$. Admitting that $K_1$ and $K_2$ are pseudo-disks, the class of all pseudo-disks is closed under contractions.
The main goal of this subsection is to prove the following result:

\begin{proposition} \label{prop:disks} A graph $D\in {\mathcal F}(Q_3)$ is a pseudo-disk if and only if $D$ is a disk.  In particular, the convex hull $\conv(C)$ of an isometric cycle $C$ of any graph $G\in {\mathcal F}(Q_3)$ is an $\AOM$ of rank $2$.
\end{proposition}

\begin{proof} The fact that disks are pseudo-disks follows from the next claim:

\begin{claim}\label{lem:AOMdisk}
 If $D\in\mathcal{F}(Q_3)$ is a disk, then $D$ is the convex hull of an isometric cycle $C$ of $D$.
\end{claim}
\begin{proof}
By definition, $D$ is the region graph of an arrangement $\mathcal A$ of pseudolines. The cycle $C$ is obtained by traversing the unbounded cells
of the arrangement in circular order, i.e., $C=\partial D$. This cycle $C$ is isometric
in $D$ because the regions corresponding to any two opposite vertices $v$ and $-v$ of $C$ are separated by all pseudolines of $\mathcal A$, thus $d_D(v,-v)=|\mathcal A|$.
Moreover, $\conv(C)=D$ because for any other vertex $u$ of $D$, any pseudoline
$\ell\in \mathcal A$ separates exactly one of the regions corresponding to $v$ and $-v$ from the region corresponding to $u$, whence $d_D(v,u)+d_D(u,-v)=d_D(v,-v)$.
\end{proof}

The remaining part of the proof is devoted to prove that any pseudo-disk is a disk. Let $D$ be a pseudo-disk with boundary $C$. Let $A_D:=\{v\in D: v \text{ has an antipode}\}$. As before, for a $\Theta$-class $E_i$ of $D$,
by $D^+_i$ and $D^-_i$ we denote the complementary halfspaces of $D$ defined by $E_i$.

\begin{claim}\label{lem:cycle}
 If $D$ is a pseudo-disk with boundary $C$, then $A_D=C$.
\end{claim}

\begin{proof} Clearly, $C\subseteq A_D$. To prove $A_D\subseteq C$, suppose by way of contradiction that $v,-v$ are antipodal vertices of $D$ not belonging to $C$.
Contract the $\Theta$-classes until $v$ is adjacent to a vertex $u\in C$, say via an edge in class $E_i$ (we can do this because all such classes crosses $C$ and
by Lemma \ref{contraction-ChKnMa}(ii) their contraction will lead to a disk). Let $u\in D^+_i$ and $v\in D^-_i$.
Since $D=\conv(C)$, the $\Theta$-class $E_i$ crosses $C$. Let $xy$ and $zw$ be the two opposite edges of $C$ belonging to $E_i$ and let $x,z\in D^+_i, y,w\in D^-_i$. Let $P,Q$ be two shortest paths in $D^-_i$
connecting $v$ with $y$ and $w$, respectively. Since the total length of $P$ and $Q$ is equal to the shortest path of $C$ from $x$ to $z$ passing through $u$, the paths $P$ and $Q$ intersect only in $v$.
Extending $P$ and $Q$, respectively within $D^-_i\cap C$ until $-u$, yields shortest paths $P', Q'$ that are crossed by all $\Theta$-classes except $E_i$. Therefore, both such paths can be extended to
shortest $(v,-v)$-paths by adding the edge $-u-v$ of $E_i$. Similarly to the case of $v$, there are shortest paths $P'', Q''$ from the vertex $-v\in D^+_i$ to the vertices $x,z\in C\cap D^+_i$. Again,
$P''$ and $Q''$ intersect only in $-v$. Let $E_j$ be any $\Theta$-class crossing $P$ and $E_k$ be any $\Theta$-class crossing $Q$. We assert that the set $S:=\{ u,v,x,y,z,w,-u,-v\}$ of vertices of $D$
shatter $\{ i,j,k\}$, i.e., that contracting all $\Theta$-classes except $E_i,E_j$, and $E_k$ yields a forbidden $Q_3$. Indeed, $E_i$ separates $S$ into the sets $\{ u,x,-v,z\}$ and $\{ v,y,-u,w\}$, $E_j$ separates $S$ into the
sets $\{ x,y,-v,-u\}$ and $\{ u,v,z,w\}$, and $E_k$ separates $S$ into the sets $\{ u,v,x,y\}$ and $\{ -v,-u,z,w\}$. This  contradiction shows that $A_D\subseteq C$, whence $A_D=C$.
\end{proof}

\begin{claim}\label{lem:affine}
 If $D$ is a pseudo-disk with boundary $C$, then $D$ is an affine partial cube. Moreover, there exists an antipodal partial cube $D'\in \mathcal{F}(Q_{4})$ containing $D$ as a halfspace. 
\end{claim}
\begin{proof} First we show that $D$ is affine. Let $u,v\in D$. Using the characterization of affine partial cubes provided by \cite[Proposition 2.16]{KnMa}
we have to show that for all vertices $u,v$ of $D$ one can find $w,-w\in A_D$ such that the intervals $I(w,u)$ and $I(v,-w)$ are not crossed by the same $\Theta$-class of $D$.
By Claim~\ref{lem:cycle} this is equivalent to finding such $w,-w$ in $C$. Let $I$ be the index set of all $\Theta$-classes crossing  $I(u,v)$. Without loss of generality assume that $u\in D_i^+$ (and therefore $v\in D^-_i$)
for all $i\in I$. We assert that $(\bigcap_{i\in I}D_i^+)\cap C\neq \emptyset$. Then any vertex from this intersection can play the role of $w$.

For $i\in I$, let $C^+_i=C\cap D^+_i$ and $C^-_i=C\cap D^-_i$; $C^+_i$ and $C^-_i$ are two disjoint shortest paths of $C$ covering all vertices of $C$.
Viewing $C$ as a circle, $C^+_i$ and $C^-_i$ are disjoint arcs of this circle.  Suppose by way of contradiction that $\bigcap_{i\in I}C_i^+=\bigcap_{i\in I}D_i^+\cap C=\emptyset$. By the Helly property for arcs of a circle,
there exist  three classes $i,j,k\in I$ such that the paths $C^+_i,C^+_j,$ and $C^+_k$ pairwise intersect, together cover all the vertices and edges of the cycle $C$, and all three have empty intersection.
This implies that $C$ is cut into 6 nonempty paths: $C_i^+\cap C_j^+\cap C_k^-$, $C_i^+\cap C_j^-\cap C_k^-$, $C_i^+\cap C_j^-\cap C_k^+$, $C_i^-\cap C_j^-\cap C_k^+$, $C_i^-\cap C_j^+\cap C_k^+$, and $C_i^-\cap C_j^+\cap C_k^-$.
Recall also that $u\in D_i^+\cap D_j^+\cap D_k^+$ and $v\in D_i^-\cap D_j^-\cap D_k^-$. But then the six paths partitioning $C$ together with $u,v$ will shatter the set $\{ i,j,k\}$, i.e.,  contracting all $\Theta$-classes except $i,j,k$
yields a forbidden $Q_3$.

Consequently, $D$ is an affine partial cube, i.e., $D$ is a halfspace of an antipodal partial cube $G$, say $D=G^+_i$ for a $\Theta$-class $E_i$. Suppose that $G$ can be contracted to the 4-cube $Q_4$. If $E_i$ is a
coordinate of $Q_4$ (i.e., the class $E_i$ is not contracted), since $D=G^+_i$, we obtain that $D$ can be contracted to $Q_3$, which is impossible because $D\in {\mathcal F}(Q_3)$. Therefore $E_i$ is contracted.
Since the contractions of $\Theta$-classes commute, suppose without loss of generality that $E_i$ was contracted last. Let $G'$ be the partial cube obtained at the step before contracting  $E_i$. Let $D'$
be the isometric subgraph of $G'$ which is the image of $D$ under the performed contractions. Since the property of being a pseudo-disk is preserved by contractions, $D'$ is a pseudo-disk, moreover $D'$ is one of the two halfspaces
of $G'$ defined by the class $E_i$ restricted to $G'$. Analogously, by Lemma \ref{antipodal-contraction}
antipodality is preserved by contractions, whence $G'$ is an antipodal partial cube such that $\pi_i(G')=Q_4$.  This implies that $G'$ was obtained from $H:=Q_4$ by an isometric antipodal
expansion $(H^1,H^0,H^2)$. Notice that one of the isometric subgraphs $H^1$ or $H^2$ of the 4-cube $H$, say $H_1$ coincides with the disk $D'':=\pi_i(D')$. Since $H$ is antipodal, by \cite[Lemma 2.14]{KnMa},
$H_0$ is closed under antipodes in $Q_4$ and $-(H_1\setminus H_0)=H_2\setminus H_0$. Since $H_0$ is included in the isometric subgraph $H_1=D''$ of $H$, $H_0$ is closed under antipodes also in $D''$.
By Claim~\ref{lem:cycle} we obtain $H_0=A_{D''}=\partial D''$. Consequently, $H_0$ is an isometric cycle of  $H=Q_{4}$ that separates $Q_{4}$ in two sets of vertices.
However, no isometric cycle of $Q_4$ separates the graph.
\end{proof}

\begin{figure}[htb]
\centering
\includegraphics[width=.25\textwidth]{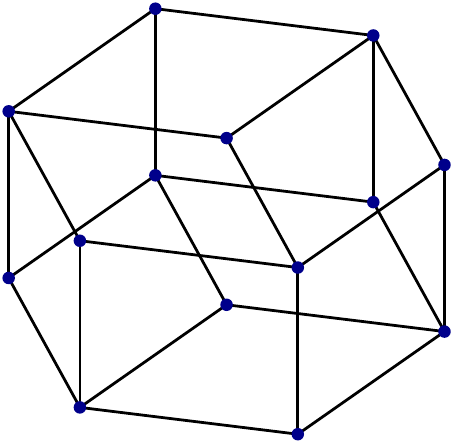}
\caption{An OM containing $Q_3^-$ as a halfspace.}
\label{fig:OM}
\end{figure}

If $D\notin\mathcal{F}(Q_3)$ is the convex hull of an isometric cycle, then $D$ is not necessarily affine, see $X_4^5$ in Figure~\ref{fig:COMobstructions}. On the other hand,
$SK_4\in\mathcal{F}(Q_3)$ is affine but is not a pseudo-disk. Let us introduce the distinguishing feature.

\begin{claim}\label{lem:diskAOM}
If $D$ is a pseudo-disk with boundary $C$, then $D$ is a disk, i.e., the region graph of a pseudoline arrangement.
\end{claim}
\begin{proof}
By Claim~\ref{lem:affine} we know that $D$ is the halfspace of an antipodal partial cube $G$. Suppose by contradiction that $G$ is not an $\OM$. By~\cite{KnMa} $G$ has a minor $X$ from the family $\mathcal{Q}^{-}$.
Since the members of this class are non-antipodal, to obtain $X$ from $G$ not only contractions but also restrictions are necessary. We perform first all contractions $I$ to obtain a pseudo-disk $D':=\pi_I(D)\in \mathcal{F}(Q_3)$
that is a halfspace of the antipodal graph $G':=\pi_I(G)$. By the second part of Claim~\ref{lem:affine} we know that $G'\in\mathcal{F}(Q_4)$. Now, since $G'$ contains $X$ as a proper convex subgraph, by Lemma~\ref{lem:antipodal}
we get $X\in\mathcal{F}(Q_3)$. Since $SK_4$ is the only member of the class $\mathcal{Q}^-$ containing $SK_4$ as a convex subgraph, by Proposition~\ref{prop:excludedminors}, we obtain $X=SK_4$.
Assume minimality in this setting, in the sense that any further contraction
destroys all copies of $X$ present in $D'$. We distinguish two cases.

First, suppose that there exists a copy of $X$ which is a convex subgraph of $D'$. Let $n\geq 4$ be maximal such that there is a convex $H=SK_n$ in $D'$ extending a convex copy of $X$. By Proposition~\ref{gatedSKn}, $H$ is gated. If $H\ne D'$, there exists a $\Theta$-class $E_i$ of $D'$ not crossing $H$. Contracting $E_i$, by Lemma \ref{contraction-ChKnMa}(iii) we will obtain a gated full subdivision $\pi_i(H)=SK_n$ contrary to the minimality in the choice of $D'$. Therefore $D'=H=SK_n$,  but it is easy to see that all $SK_n, n\ge 4,$ are not pseudo-disks, a contradiction.

Now,  suppose that no copy of $X$ is  a convex subgraph of $D'$. Since $G'$ contains $X$ as a convex subgraph, $D'$ is a halfspace of $G'$ (say $D'=(G')^+_i$) defined by a $\Theta$-class $E_i$, and $G'$ is an antipodal partial cube, we conclude that $E_i$ crosses all convex copies $H$ of $X=SK_4$.
Then $E_i$ partitions $H$ into a 6-cycle $C$ and  a $K_{1,3}$ such that all edges between them belong to $E_i$.
The antipodality map of $G'$  maps the vertices of $(G')^+_i$ to vertices of $(G')^-_i$ and vice-versa. Therefore in $D'$  there must be a copy of $K_{1,3}$ and a copy of $C=C_6$, and both such copies belong to the boundary $\partial (G')^+_i$.
The antipodality map is also edge-preserving. Therefore, it maps edges of $E_i$ to edges of $E_i$ and vertices of $(G')^+_i\setminus \partial (G')^+_i$
to vertices of $(G')^-_i\setminus \partial (G')^-_i$. Consequently,  all vertices of $\partial (G')^-_i$
have antipodes in the pseudo-disk $D'=(G')^+_i$ and their antipodes also belong to $\partial (G')^+_i$. This and Claim~\ref{lem:cycle} imply that
$\partial (G')^+_i\subset A_{D'}=\partial D'$. Therefore the isometric cycle $\partial D'$ contains an isometric copy of $C_6$, whence $\partial D'=C_6$. Since $\partial D'$ also contains the leafs of a $K_{1,3}$ we conclude that the pseudo-disk
$D'$ coincides with $Q^-_3$.  However, the only antipodal partial cube containing $Q_3^-$ as a halfspace is depicted in Figure~\ref{fig:OM} and it is an $\OM$, leading to a contradiction.
\end{proof}

Note that Claim~\ref{lem:diskAOM} generalizes Lemma~\ref{6cycle}. Together with Claim~\ref{lem:AOMdisk} it yields that pseudo-disks are disks, i.e., tope graphs of $\AOM$s of rank two,  concluding the proof of Proposition \ref{prop:disks}.
\end{proof}

\subsection{Gated hulls of long isometric cycles}

By Proposition \ref{prop:disks} disks and pseudo-disks are the same, therefore, from now on we use the name ``disk'' for both.
We continue by showing that in two-dimensional partial cubes all disks with boundary of length $>6$  are gated.

\begin{proposition} \label{gatedcycle} If $D$ is a disk  of $G\in {\mathcal F}(Q_3)$ and $|\partial D|>6$, then $D$ is a gated subgraph of $G$.
In particular, convex long cycles of $G$ are gated.
\end{proposition}

\begin{proof} Let $G$ be a minimal two-dimensional partial cube in which the assertion does not hold. Let $D$ be a non-gated disk of $G$ whose boundary $C:=\partial D$ is a long isometric cycle.
Let $v$ be a vertex of $G$ that has no gate in $D$ and is as close as possible to $D$, where $d_G(v,D)=\min \{ d_G(v,z): z\in D\}$. 
We  use some notations from the proof of \cite[Proposition 1]{ChKnMa}. Let $P_v:=\{ x\in D: d_G(v,x)=d_G(v,D)\}$ be the \emph{metric projection}
of $v$ to $D$. Let also $R_v:=\{ x\in D: I(v,x)\cap D=\{ x\}\}.$ Since $D$ is not gated, $R_v$ contains at least two vertices.
Obviously, $P_v\subseteq R_v$ and the vertices of $R_v$ are pairwise nonadjacent. We denote the vertices of $P_v$ by $x_1,\ldots,x_k$. For any  $x_i\in P_v$, let
$v_i$ be a neighbor of $v$ on a shortest $(v,x_i)$-path. By the choice of $v$, each  $v_i$ has a gate in $D$. By the definition
of $P_v$, $x_i$ is the gate of $v_i$ in $D$. This implies that the vertices $v_1,\ldots,v_k$ are pairwise distinct. Moreover, since
$x_i$ is the gate of $v_i$ in $D$, for any two distinct vertices $x_i,x_j\in P_v$, we have $d_G(v_i,x_i)+d_G(x_i,x_j)=d_G(v_i,x_j)\le 2+d_G(v_j,x_j)$.
Since $d_G(x_i,v_i)=d_G(x_j,v_j)$, necessarily $d_G(x_i,x_j)=2$.

We assert that any three distinct vertices $x_j,x_k,x_\ell \in P_v$ do not have a common neighbor. Suppose by way of contradiction that there exists a vertex $x$ adjacent to $x_j,x_k,x_\ell$. Then $x$ belongs to $D$ by convexity of $D$ and
$x_j,x_k,x_\ell\in I(x,v)$ since $x_j,x_k,x_\ell \in P_v$. Let $E_j$ be the $\Theta$-class of the edge $v_jv$ and let $C_k$ be the cycle of $G$ defined by a $(v,x_j)$-shortest path $P$ passing via $v_j$, the 2-path $(x_j,x,x_k)$, and a shortest $(x_k,v)$-path $Q$ passing via $v_k$. Then $E_j$ must contain another edge of $C_k$. Necessarily this cannot be an edge of $P$. Since $v$ is a closest vertex to $D$ without a gate, this edge cannot be an edge of $Q$. Since $x_j\in I(x,v)$, this edge is not $xx_j$. Therefore the second edge of $E_j$ in $C_k$ is the edge $xx_k$. This implies that $v$ and $x_k$ belong to the same halfspace defined by $E_j$, say $G^+_j$, and $v_j$ and $x$ belong to its complement $G^-_j$. Using an analogously defined cycle $C_{\ell}$, one can show that the edge $xx_{\ell}$ also belong to $E_j$, whence the vertices $x_k$ and $x_{\ell}$ belong to the same halfspace $G^+_j$.
Since $x\in I(x_k,x_{\ell})$ and $x\in G_j^-$, we obtain a contradiction with convexity of $G^+_j$. Therefore, if $x_j,x_k,x_\ell \in P_v$, then $\conv(x_j,x_k,x_\ell)$ is an isometric 6-cycle of $D$. In particular, this implies that each of the intervals $I(x_j,x_k),I(x_k,x_{\ell}), I(x_j,x_{\ell})$ consists of a single shortest path.

Next we show that $|P_v|\le 3$. Suppose by way of contradiction that $|P_v|\ge 4$ and pick the vertices $x_1,x_2,x_3,x_4\in P_v$. Let $H$ be the subgraph of $D$ induced by the union of the intervals $I(x_j,x_k)$, with $j,k\in \{ 1,2,3,4\}$. Since these intervals are 2-paths intersecting only in common end-vertices, $H$ is isomorphic to $SK_4$ with $x_1,x_2,x_3,x_4$ as original vertices. Since $D$ is a two-dimensional partial cube, one can directly check that $H$ is an isometric subgraph of $D$. Since the intervals $I(x_j,x_k)$ are interiorly disjoint paths, $H=SK_4$ cannot be extended to $SK_4^*$. By Claim \ref{convexSKn}, $H=SK_4$ is a convex subgraph of $D$. Since $D$ is an AOM of rank 2 and thus a COM of rank 2, by Proposition \ref{prop:excludedminors}, $D$ cannot contain $SK_4$ as a pc-minor. This contradiction shows that $|P_v|\le 3$.


Let $S:= \conv(P_v)$. Since $|P_v|\le 3$ and $d_G(x_j,x_k)=2$ for any two vertices $x_j,x_k$ of $P_v$, there exists at most three $\Theta$-classes crossing $S$. Since the length of the isometric cycle $C$ is at least 8, there exists a $\Theta$-class
$E_i$ crossing $C$ (and $D$) and not crossing $S$.  We assert that $v$ and the vertices of $P_v$  belong to the same halfspace defined by $E_i$. Indeed, if $E_i$ separates $v$ from $S$, then for any $j$, $E_i$ has an edge on any shortest $(v_j,x_j)$-path. This contradicts the fact that $x_j$ is the gate of $v_j$ in $D$. Consequently, $v$ and the set $S$ belong to the same halfspace defined by $E_i$.
Consider the graphs $G':=\pi_i(G)$, $D':=\pi_i(D)$ and the cycle $C':=\pi_i(C)$.  By Lemma \ref{contraction-ChKnMa}(i), $D'$ is a disk with boundary $C'$ (and thus an $\AOM$) of the two-dimensional partial cube $G'$.
Notice that the distance in $G'$ between $v'$ and the vertices $x'_j$ of $P_v$ is the same as the distance
between $v$ and $x_j$ in $G$ and that the distance between $v'$ and the images of vertices of $R_v\setminus P_v$ may eventually decrease by 1.
This implies that $D'$ is not gated. By minimality of $G$, this is possible only if $C'$ is a 6-cycle. In this case, by Proposition \ref{gatedhullC6}, we conclude that $D'$ is included in a maximal full subdivision $H'=SK_n$, which is a gated subgraph of $G'$.
The graph $G$ is obtained from $G'$ by an isometric expansion $\psi_i$ (inverse to $\pi_i$).
By Lemma \ref{expansionSKn}, $\psi_i$, restricted to $H'$, is a peripheral expansion along an isometric tree of $H'$. This implies that in $G$ the convex $\AOM$ $D$ is contained in a full
subdivision of $K_n$, contrary to the assumption that $D$ is the convex hull of the isometric cycle $C$ of length at least 8.
\end{proof}

Summarizing  Propositions \ref{gatedhullC6},  \ref{prop:disks}, and \ref{gatedcycle}, we obtain the following results:

\begin{theorem} \label{FS+AOM} Let $G$ be a two-dimensional partial cube and $C$ be an isometric cycle of $G$. If $C=C_6$, then the gated hull of $C$ is either $C$, $Q^-_3$, or a maximal full subdivision. If otherwise $C$ is long, then $\conv(C)$ is a gated disk.
\end{theorem}

\begin{corollary} Maximal full subdivisions, convex disks with long cycles as boundaries (in particular, long convex cycles) are gated subgraphs in two-dimensional partial cubes.
\end{corollary}

\section{Completion to ample partial cubes} In this section, we prove that any partial cube $G$ of VC-dimension 2
can be completed to an ample partial cube $G^{\top}$ of VC-dimension 2. We perform this completion in two steps. First, 
we canonically extend $G$ to a partial cube $G\urcorner\in {\mathcal F}(Q_3)$ not containing convex full subdivisions.
The resulting graph $G\urcorner$ is a COM of rank 2: its cells are the gated cycles of $G$ and the 4-cycles created by extensions of full subdivisions.
Second, we transform $G\urcorner$ into an ample partial cube $(G\urcorner)\ulcorner\in {\mathcal F}(Q_3)$ by filling each gated cycle $C$
of length $\ge 6$ of $G$ (and of $G\urcorner$) by a planar tiling with squares. Here is the main result of this
section and one of the main results of the paper:

\begin{theorem}\label{thm:extendtoample}
Any $G\in\mathcal{F}(Q_3)$ can be completed to an ample partial cube $G^{\top}:=(G\urcorner)\ulcorner\in\mathcal{F}(Q_3)$.
\end{theorem}

\subsection{Canonical completion to two-dimensional COMs}

The {\it 1-extension} graph of a partial cube $G\in {\mathcal F}(Q_3)$ of $Q_m$  is a subgraph $G'$ of $Q_m$ obtained by taking a maximal by inclusion convex full subdivision $H=SK_n$ of $G$ such that $H$ is standardly embedded in $Q_m$ and adding to $G$ the vertex $v_{\varnothing}$.

\begin{lemma} \label{1extension} If $G'$ is the 1-extension of $G\in {\mathcal F}(Q_3)$ and $G'$ is obtained with respect to the maximal by inclusion  convex full subdivision $H=SK_n$ of $G$, then $G'\in {\mathcal F}(Q_3)$ and $G$ is an isometric subgraph of $G'$. Moreover, any convex full subdivision $SK_{r}$ with $r\ge 3$ of $G'$ is a convex full subdivision of $G$ and any convex cycle of length $\ge 6$ of $G'$ is a convex cycle of $G$.
\end{lemma}

\begin{proof} Let $G$ be an isometric subgraph of $Q_m$. To show that $G'$ is an isometric subgraph of $Q_m$ it suffices to show that any vertex $v$ of $G$ can be connected in $G'$ with $v_{\varnothing}$ by a shortest path.
By Proposition \ref{gatedSKn} $H$ is a gated subgraph of $G$ and the gate $v_0$ of $v$ in $Q=\conv(H)$ belongs to $H$. This means that if $v$ is encoded by the set $A$ and $v_0$ is encoded by the set $A_0=A\cap \{ 1,\ldots,n\}$,
then either $A_0=\{ i\}$ or $A_0=\{ i,j\}$ for an original vertex $u_i$ or a subdivision vertex $u_{i,j}$.  This means that $d(v,v_0)=d(v,u_i)=|A|-1$ in the first case and $d(v,v_0)=d(v,u_{i,j})=|A|-2$ in the second case.
Since $d(v,v_{\varnothing})=|A|$, we obtain a shortest $(v,v_{\varnothing})$-path  in $G'$ first going from $v$ to $v_0$ and then from $v_0$ to $v_{\varnothing}$ via an edge or a path of length 2 of $H$.
This establishes that $G'$ is an isometric subgraph of $Q_m$. Since any two neighbors of $v_{\varnothing}$ in $H$ have distance 2 in $G$ and $v_{\varnothing}$ is adjacent in $G$ only to the original vertices of $H$, we also conclude that
$G$ is an isometric subgraph of $G'$.

Now we will show that $G'$ belongs to ${\mathcal F}(Q_3)$. Suppose by way of contradiction that the sets corresponding to some set $S$ of  8 vertices of $G'$ shatter  the set $\{ i,j,k\}$. Since $G\in {\mathcal F}(Q_3)$, one of the vertices of $S$ is the vertex $v_{\varnothing}$: namely, $v_{\varnothing}$ is the vertex whose trace on $\{ i,j,k\}$ is $\varnothing$. Thus the sets corresponding to the remaining 7 vertices of $S$ contain at least one of the elements $i,j,k$. Now, since $H=SK_n$ with $n\ge 4$, necessarily there exists an original vertex $u_{\ell}$ of $H$ with $\ell\notin \{ i,j,k\}$. Clearly, $u_{\ell}$ is not a vertex of $S$. Since the trace of $\{ \ell\}$ on $\{ i,j,k\}$ is $\varnothing$, replacing in $S$ the vertex $v_{\varnothing}$ by $u_{\ell}$ we will obtain a set of 8 vertices of $G$ still shattering the set $\{ i,j,k\}$, contrary to $G\in {\mathcal F}(Q_3)$.

It remains to show that any convex full subdivision of $G'$ is a convex full subdivision of $G$. Suppose by way of contradiction that $H'=SK_r, r\ge 3,$ is a convex full subdivision of $G'$ containing the vertex $v_{\varnothing}$. By Claim \ref{claimSKn}, in $G'$ $v_{\varnothing}$ is adjacent only to the original vertices of $H$. Hence, if $v_{\varnothing}$ is an original vertex of $H'$ then at least two original vertices of $H$ are subdivision vertices of $H'$ and if  $v_{\varnothing}$ is a subdivision vertex of $H'$ then the two original vertices of $H'$ adjacent to $v_{\varnothing}$ are original vertices of $H$. In both cases, denote those two original vertices of $H$ by $x=u_i$ and $y=u_j$. Since $H'$ is convex and $u_{i,j}$ is adjacent to $u_i$ and $u_j$, $u_{i,j}$ must belong to $H'$. But this implies that $H'$ contains the 4-cycle $(x=u_i,v_{\varnothing},y=u_j,u_{i,j})$, which is impossible in a convex full subdivision. In a similar way, using Claim \ref{claimSKn}, one can show that any convex cycle of length $\ge 6$ of $G'$ is a convex cycle of $G$.
\end{proof}


Now, suppose that we consequently perform the operation of 1-extension to all gated full subdivisions and to the occurring intermediate partial cubes. By Lemma \ref{1extension} all such isometric subgraphs of $Q_m$ have VC-dimension 2 and all occurring convex full subdivisions  are already convex full subdivisions of $G$. After a finite number of 1-extension steps (by the Sauer-Shelah-Perles lemma, after at most $\binom{m}{\leq 2}$ 1-extensions), we will get an isometric subgraph $G\urcorner$ of $Q_m$ such that $G\urcorner\in {\mathcal F}(Q_3)$, $G$ is an isometric subgraph of $G\urcorner$, and all maximal full subdivisions $SK_n$ of $G\urcorner$ are included in $SK^*_n$. We call $G\urcorner$ the {\it canonical 1-completion} of $G$.   We summarize this result in the following proposition:

\begin{proposition}\label{prop:extendtoCOM} If $G\in {\mathcal F}(Q_3)$ is an isometric subgraph of the hypercube $Q_m$, then after at most $\binom{m}{\leq 2}$ 1-extension steps, $G$ can be canonically completed to
a two-dimensional COM  $G\urcorner$ and $G$ is an isometric subgraph of  $G\urcorner$.
\end{proposition}

\begin{proof} To prove that $G\urcorner$ is a two-dimensional COM, by second assertion of Proposition \ref{prop:excludedminors} we have to prove that $G\urcorner$ belongs to $\mathcal{F}(Q_3, SK_4)=\mathcal{F}(Q_3)\cap \mathcal{F}(SK_4)$.
The fact that $G\urcorner$ belongs to $\mathcal{F}(Q_3)$ follows from Lemma \ref{1extension}. Suppose now that $G\urcorner$ contains $SK_4$ as a pc-minor. By Corollary \ref{SKnCopie}, $G\urcorner$ contains a convex subgraph $H$ isomorphic to $SK_4$.  Then $H$ extends in $G\urcorner$ to a maximal by inclusion $SK_n$, which we denote by $H'$. Since $G\urcorner\in \mathcal{F}(Q_3)$ and $H$ does not extend to $SK^*_4$, $H'$ does not extend to $SK^*_n$ either. By Claim \ref{convexSKn} and Proposition \ref{gatedSKn} applied to $G\urcorner$, we conclude that $H'$ is a convex and thus gated subgraph of $G\urcorner$. Applying the second assertion of Lemma \ref{1extension} (in the reverse order) to all pairs of graphs occurring in the construction transforming $G$ to $G\urcorner$, we conclude that $H'$ is a convex and thus gated full subdivision of $G$. But this is impossible because all maximal full subdivisions $SK_n$ of $G\urcorner$ are included in $SK^*_n$. This shows that $G\urcorner$ belongs to $\mathcal{F}(SK_4)$, thus $G\urcorner$ is a two-dimensional COM. That $G$ is isometrically embedded in $G\urcorner$ follows from Lemma \ref{1extension} and the fact that if $G$ is an isometric subgraph of $G'$ and $G'$ is an isometric subgraph of $G''$, then $G$ is an isometric subgraph of $G''$.
\end{proof}%

\subsection{Completion to ample two-dimensional partial cubes}

Let $G\in {\mathcal F}(Q_3)$, $C$ a gated cycle of $G$, and $E_j$ a $\Theta$-class crossing $C$. Set $C:=(v_1,v_2,\ldots, v_{2k})$, where the edges $v_{2k}v_1$ and $v_kv_{k+1}$ are in $E_j$. The graph $G_{C,E_j}$ is defined by adding a path on vertices $v_{2k}=v'_1, \ldots, v'_{k}=v_{k+1}$ and edges $v_iv'_i$ for all $2\leq i\leq k-1$. Let  $C'=(v'_1,\ldots,v'_k,v_{k+2},\ldots, v_{2k-1})$. Then we recursively apply the same construction to the cycle $C'$ and we call the resulting graph
a \emph{cycle completion} of $G$ \emph{along a gated cycle $C$}; see Figure~\ref{fig:extendcycle} for an illustration. Proposition \ref{prop:extendcycle} establishes the basic properties of this construction, in particular it shows that the cycle completion
along a gated cycle is well defined.

\begin{figure}[htb]
\centering
\includegraphics[width=.80\textwidth]{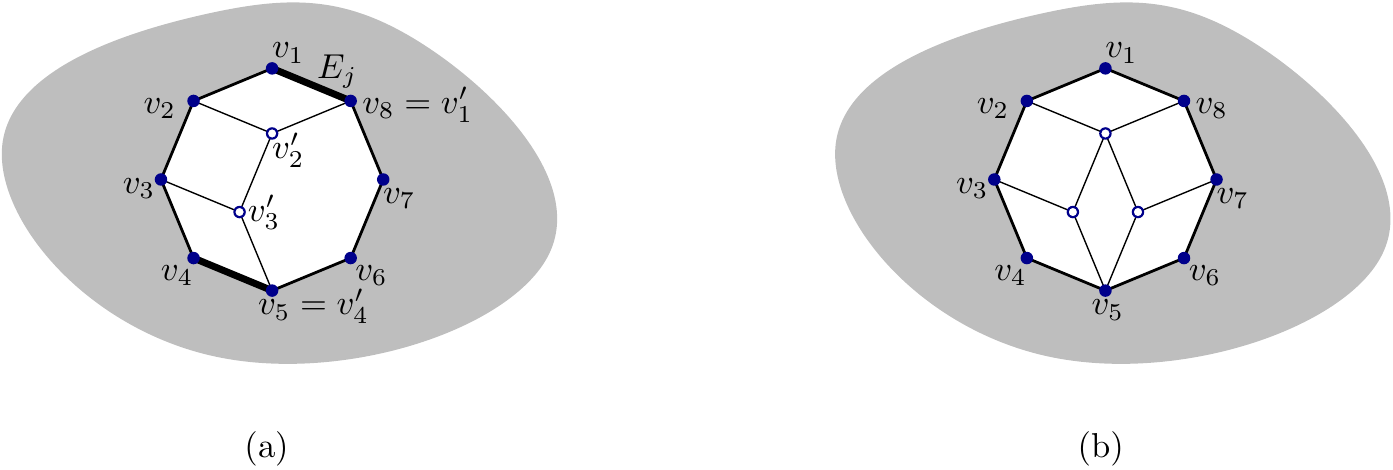}
\caption{(a) $G_{C,E_j}$ is obtained by adding the white vertices to a graph $G$ with a gated cycle $C=(v_1,v_2,\ldots, v_{8})$. (b) A cycle completion of $G$ along the cycle $C=(v_1,v_2,\ldots, v_{8})$.}
\label{fig:extendcycle}
\end{figure}

\begin{proposition}\label{prop:extendcycle}
 Let $G$ be a partial cube, $C$ a gated cycle of $G$, and $E_j$ a $\Theta$-class crossing $C$.
 \begin{enumerate}[(1)]
  \item \label{cond1:prop_extendcycle} $G_{C,E_j}$ is a partial cube and $G$ is an isometric subgraph of $G_{C,E_j}$,
  \item \label{cond2:prop_extendcycle} $C'=(v'_1,\ldots,v'_k,v_{k+2},\ldots, v_{2k-1})$ is a gated cycle,
  \item \label{cond3:prop_extendcycle} If $G\in {\mathcal F}(Q_3)$, then so is $G_{C,E_j}$,
  \item \label{cond4:prop_extendcycle} If $G$ contains no convex $SK_n$, then neither does $G_{C,E_j}$.
\end{enumerate}
\end{proposition}

\begin{proof}


To prove \eqref{cond1:prop_extendcycle}, notice that the $\Theta$-classes of $G$ extend to $G_{C,E_j}$ is a natural way, i.e., edges of the form $v_iv'_i$ for all $2\leq i\leq k-1$ belong to $E_j$, while an edge $v'_iv'_{i+1}$ belongs to the $\Theta$-class of the edge $v_iv_{i+1}$ for all $1\leq i\leq k-1$. Clearly, among the old vertices distances have not changed and the new vertices are embedded as an isometric path. If $w\in C$ and $u\in C'$ is a new vertex, then it is easy to see that there is a shortest path using each $\Theta$-class at most once. In fact, since $w$ is at distance at most one from $C'$ it has a gate in $C'$, i.e., the path only uses $E_j$. Finally, let $v$ be an old vertex of $G \setminus C$, $w$ be its gate in $C$, and $u$ be a new vertex, i.e., $u\in G_{C, E_j} \setminus G$. Let $P$ be a path from $v$ to $u$ that is a concatenation of a shortest $(v,w)$-path $P_1$ and a shortest $(w,u)$-path $P_2$.
Since $C$ is gated and all $\Theta$-classes crossing $P_2$ also cross $C$, the $\Theta$-classes of $G$ crossing $P_1$ and the $\Theta$-classes  crossing $P_2$ are distinct. Since $P_1$ and $P_2$ are shortest paths, the $\Theta$-classes in each of two groups are also pairwise different. Consequently, $P$ is a shortest $(v,u)$-path and thus $G_{C,E_j}$ is a partial cube. Finally, $G$ is an isometric subgraph of $G_{C,E_j}$ by construction.

To prove \eqref{cond2:prop_extendcycle}, let $v \in G \setminus C'$. If $v \in G \setminus C$, let $w$ be its gate in $C$. Thus there is a shortest $(v,w)$-path which does not cross the $\Theta$-classes crossing $C$. Suppose that $w \notin C'$, otherwise we are done. Then there exists a vertex $w'$ such that the edge $ww'$ belongs to $E_j$. Since $E_j$ crosses $C$ and not $C'$, $w'$ is the gate of $v$ in $C'$. If $v \in C \setminus C'$, using the previous argument, there exists an edge $vv'$ belonging to $E_j$ and we conclude that $v'$ is the gate of $v$ in $C'$.

To prove \eqref{cond3:prop_extendcycle}, suppose by way of contradiction that $G_{C,E_j}$ has a $Q_3$ as a pc-minor. Then there exists a sequence $s$ of restrictions $\rho_s$ and contractions $\pi_s$ such that $s(G) = Q_3$.
Recall that restrictions and contractions commute in partial cube \cite{ChKnMa}. Hence, we get a graph $G'=\pi_s(G)$ which contains a convex $Q_3$. Thus, this pc-minor $Q_3$ can be obtained by contractions.
Clearly, $E_j$ must be among the uncontracted classes, because $\pi_j(G_{C,E_j})=\pi_j(G)$. Furthermore, if only one other $\Theta$-class of $C$ is not contracted in $G_{C,E_j}$, then contraction will identify all new vertices with (contraction) images of old vertices and again by the assumption
$G\in {\mathcal F}(Q_3)$ we get a contradiction. Thus, the three classes that constitute the copy of $Q_3$ are $E_j$ and two other classes say $E_j', E_j''$ of $C$. Thus, the augmented $C$ yields a $Q_3^-$ in the contraction of $G_{C,E_j}$,
but the last vertex of the $Q_3$ comes from a part of $G$. In other words, there is a vertex $v\in G$, such that all shortest paths from $v$ to $C$ cross $E_j$, $E_j'$, or $E_j''$. This contradicts that $C$ was gated, establishing that
$G_{C,E_j}\in {\mathcal F}(Q_3)$.

To prove \eqref{cond4:prop_extendcycle}, suppose by way of contradiction that $G_{C,E_j}$ contains a convex $SK_n$. Since $SK_n$ has no 4-cycles nor vertices of degree one, the restrictions leading to $SK_n$ must either include $E_j$ or the class of the edge $v_1v_2$ or $v_{2k-1}v_{2k}$. The only way to restrict here in order to obtain a graph that is not a convex subgraph of $G$ is restricting to the side of $E_j$, that contains the new vertices. But the obtained graph cannot use new vertices in a convex copy of $SK_n$ because they form a path of vertices of degree two, which does not exist in a $SK_n$. Thus $G_{C,E_j}$ does not contain a convex $SK_n$.
\end{proof}


Propositions \ref{prop:extendtoCOM} and \ref{prop:extendcycle} allow us to prove Theorem \ref{thm:extendtoample}. Namely, applying Proposition \ref{prop:extendtoCOM} to a graph $G \in \mathcal{F}(Q_3)$,
we obtain a two-dimensional COM $G\urcorner$, i.e. a graph $G\urcorner\in\mathcal{F}(Q_3,SK_4)$.
Then, we recursively apply the cycle completion along gated cycles to the graph $G\urcorner$ and to the graphs resulting from $G\urcorner$. By Proposition~\ref{prop:extendcycle} \eqref{cond3:prop_extendcycle}, \eqref{cond4:prop_extendcycle}, all intermediate graphs belong to $\mathcal{F}(Q_3,SK_4)$, i.e. they are two-dimensional COMs. This explain why we can recursively apply the cycle completion construction cycle-by-cycle. Since this construction does not increase the VC-dimension, by Sauer-Shelah lemma after a finite number of steps, we will get a graph $(G\urcorner)\ulcorner\in\mathcal{F}(Q_3,SK_4)$ in which all convex cycles must be gated (by Propositions \ref{gatedhullC6} and \ref{gatedcycle}) and must have length $4$. This implies that $(G\urcorner)\ulcorner\in\mathcal{F}(C_6)$. Consequently, $(G\urcorner)\ulcorner\in\mathcal{F}(Q_3,C_6)$ and by Proposition \ref{prop:excludedminors} the final graph $G^{\top}=(G\urcorner)\ulcorner$ is a two-dimensional ample partial cube. This completes the proof of Theorem \ref{thm:extendtoample}. For an illustration, see Figure \ref{fig:completionM}.

\begin{figure}[htb]
\centering
\includegraphics[width=.40\textwidth]{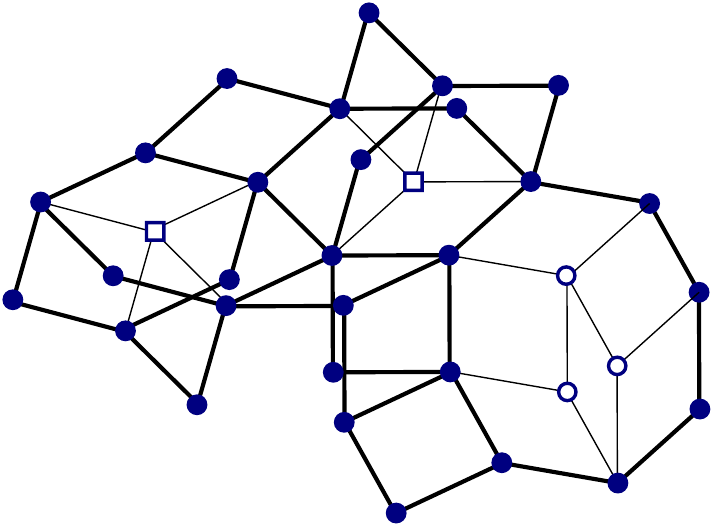}
\caption{An ample completion $M^{\top}$ of the running example $M$.}
\label{fig:completionM}
\end{figure}

\begin{remark}
 One can generalize the construction in Proposition~\ref{prop:extendcycle} by replacing a gated cycle $C$ by a gated $\AOM$ that
 is the convex hull of $C$, such that all its convex cycles are gated. In a sense, this construction captures the set of all possible extensions of the graph $G$.
\end{remark}

\section{Cells and carriers} This section uses concepts and techniques developed for COMs \cite{BaChKn} and  for hypercellular graphs \cite{ChKnMa}.
Let ${\mathcal C}(G)$ denote the set of all convex cycles of a partial cube $G$ and let ${\mathbf C}(G)$ be the 2-dimensional cell complex whose 2-cells are obtained by replacing
each convex cycle $C$ of length $2j$ of $G$ by a regular Euclidean polygon $[C]$ with $2j$ sides. It was shown in \cite{KlSh} that the set ${\mathcal C}(G)$ of convex cycles of
any partial cube $G$ constitute a basis of cycles. This result was extended in \cite[Lemma 13]{ChKnMa} where it has been shown that the 2-dimensional
cell complex ${\mathbf C}(G)$ of any partial cube $G$ is simply connected.  Recall that a cell complex $\bf X$ is {\it simply connected} if it is connected and if every continuous
map of the 1-dimensional sphere $S^1$ into $\bf X$ can be extended to a continuous mapping of the (topological) disk $D^2$ with boundary $S^1$ into $\bf X$.

Let $G$ be a partial cube. For a $\Theta$-class $E_i$ of $G$, we denote by $N(E_i)$ the \emph{carrier} of $E_i$ in ${\mathbf C}(G)$, i.e., the subgraph of $G$
induced by the union of all cells of ${\mathbf C}(G)$ crossed by $E_i$. The carrier $N(E_i)$ of $G$ splits into its positive and negative parts
$N^+(E_i):=N(E_i)\cap G^+_i$ and $N^-(E_i):=N(E_i)\cap G^-_i$, which we call {\it half-carriers}. Finally, call $G^+_i\cup N^-(E_i)$ and $G^-_i\cup N^+(E_i)$
the {\it extended halfspaces} of $E_i$.  By Djokovi\'c's Theorem \ref{Djokovic}, halfspaces of partial cubes $G$ are convex subgraphs and therefore are
isometrically embedded in $G$. However, this is no longer true for carriers, half-carriers, and extended halfspaces of all partial cubes.  However this is the case for
two-dimensional partial cubes:

\begin{proposition} \label{carriers} If $G\in {\mathcal F}(Q_3)$ and $E_i$ is a $\Theta$-class of $G$, then the carrier $N(E_i)$, its halves $N^+(E_i), N^-(E_i)$, and the extended halfspaces
$G^+_i\cup N^-(E_i), G^-_i\cup N^+(E_i)$ are isometric subgraphs of $G$, and thus belong to ${\mathcal F}(Q_3)$.
\end{proposition}

\begin{proof} Since the class ${\mathcal F}(Q_3)$ is closed under taking isometric subgraphs, it suffices to show that each of the mentioned subgraphs is an isometric subgraph of $G$.
The following claim reduces  the isometricity of carriers and extended halfspaces to isometricity of half-carriers:

\begin{claim} \label{half-carrier} Carriers and extended halfspaces of a partial cube $G$ are isometric subgraphs of $G$ if and only if half-carriers are isometric subgraphs of $G$.
\end{claim}

\begin{proof} One direction is implied by the equality $N^+(E_i):=N(E_i)\cap G^+_i$ and the general fact that the intersection of a convex subgraph and an isometric subgraph of
$G$ is an isometric subgraph of $G$. Conversely, suppose that $N^+(E_i)$ and $N^-(E_i)$ are isometric subgraphs of $G$ and we want to prove that the carrier $N(E_i)$ is isometric
(the proof for  $G^+_i\cup N^-(E_i)$ and $G^-_i\cup N^+(E_i)$ is similar). Pick any two vertices $u,v\in N(E_i)$. If $u$ and $v$ belong to the same half-carrier, say $N^+(E_i)$, then they
are connected in $N^+(E_i)$ by a shortest path and we are done. Now, let $u\in N^+(E_i)$ and $v\in N^-(E_i)$. Let $P$ be any shortest $(u,v)$-path of $G$. Then necessarily $P$ contains an edge
$u',v'$ with $u'\in \partial G^+_i\subseteq N^+(E_i)$ and $v'\in \partial G^-_i\subseteq N^-(E_i)$. Then $u,u'$ can be connected in $N^+(E_i)$ by a shortest path $P'$ and $v,v'$ can be connected
in $N^-(E_i)$ by a shortest path $P''$. The path $P'$, followed by the edge $u'v'$, and by the path $P''$ is a shortest $(u,v)$-path included in $N(E_i)$.
\end{proof}

By Claim \ref{half-carrier} it suffices to show that the half-carriers $N^+(E_i)$ and $N^-(E_i)$ of a two-dimensional partial cube $G$ are isometric subgraphs of $G$.
By Proposition \ref{prop:extendtoCOM}, $G$ is an isometric subgraph of its canonical COM-extension $G\urcorner$.  Moreover from the construction of $G\urcorner$ it follows
that the carrier $N(E_i)$ and its half-carriers $N^+(E_i)$ and $N^-(E_i)$ are subgraphs of the carrier $N\urcorner(E_i)$ and its half-carriers $N\urcorner^+(E_i), N\urcorner^-(E_i)$ in the graph $G\urcorner$.
By \cite[Proposition 6]{BaChKn}, carriers and their halves of COMs are also COMs. Consequently, $N\urcorner^+(E_i)$ and $N\urcorner^-(E_i)$ are isometric subgraphs of $G\urcorner$.
Since the graph $G\urcorner$ is obtained from $G$ via a sequence of 1-extensions, it easily follows that any shortest
path $P\subset N\urcorner^+(E_i)$  between two vertices of $N^+(E_i)$  can be replaced by a path $P'$ of the same length lying entirely in $N^+(E_i)$. Therefore $N^+(E_i)$ is an isometric subgraph
of the partial cube $N\urcorner^+(E_i)$, thus the half-carrier $N^+(E_i)$ is also an isometric subgraph of $G$.
\end{proof}

A partial cube $G=(V,E)$ is a {\it 2d-amalgam} of
two-dimensional partial cubes $G_1=(V_1,E_1), G_2=(V_2,E_2)$ both isometrically embedded in the cube $Q_m$ if the following conditions are satisfied:
\begin{itemize}
\item[(1)] $V=V_1\cup V_2, E=E_1\cup E_2$ and $V_2\setminus V_1,V_1\setminus V_2,V_1\cap V_2\ne \varnothing;$
\item[(2)] the subgraph $G_{12}$ of $Q_m$ induced by $V_1\cap V_2$ is a two-dimensional partial cube and each maximal full subdivision $SK_n$ of $G_{12}$ is maximal in $G$;
\item[(3)] $G$ is a partial cube.
\end{itemize}

As a last ingredient for the next proposition we need a general statement about COMs.
\begin{lemma}\label{lem:rankofcell}
 If $G$ is a COM and the cube $Q_d$ is a pc-minor of $G$, then there is an antipodal subgraph $H$ of $G$ that has $Q_d$ as a pc-minor.
\end{lemma}
\begin{proof}
By \cite[Lemma 6.2.]{KnMa}, if $H$ is an antipodal subgraph of a COM $G$ and $G'$ is an expansion of $G$, then the expansion $H'$ of $H$ in $G'$ is either antipodal as well or is peripheral, where the latter implies that $H'$ contains $H$ as a convex subgraph.
 In either case $G'$ contains an antipodal subgraph, that has $H$ as minor. Since $Q_d$ is antipodal, considering a sequence of expansions from $Q_d=G_0, \ldots G_k=G$ every graph at an intermediate step contains an antipodal subgraph having $Q_d$ as a minor.
\end{proof}

\begin{proposition} \label{amalgam} Two-dimensional partial cubes are obtained via successive 2d-amalgamations from their
gated cycles and gated full subdivisions. Conversely, the 2d-amalgam of two-dimensional partial cubes $G_1=(V_1,E_1)$ and $G_2=(V_2,E_2)$ of $Q_m$ is a two-dimensional partial cube of $Q_m$
in which every gated cycle or gated full subdivision belongs to at least one of the two constituents.
\end{proposition}

\begin{proof} Let $G=(V,E)$ be a two-dimensional partial cube  which is not a single cell. We can suppose that $G$ is 2-connected,
otherwise we can do an amalgam along an articulation vertex. We assert that $G$ contains two gated cells intersecting in an edge. Since the intersection of two gated sets is gated and any cell does not contain any proper
gated subgraph except vertices and edges, the intersection of any two cells of $G$ is either empty, a vertex, or an edge. If the last case never occur, since any convex cycle of $G$ is included in a single cell,
any cycle of $G$ containing edges of several cells (such a cycle exists because $G$ is 2-connected) cannot be written as a modulo 2 sum of convex cycles. This contradicts the result of \cite{KlSh} that the set
of convex cycles of any partial cube $G$ constitute a basis of cycles.  Pick two gated cells $C_1$ and $C_2$ intersecting in an edge $e$. Let $E_i$ be a $\Theta$-class crossing $C_1$ and not containing $e$. Since $C_2$ is gated, $C_2$ is contained in
one of the halfspaces $G^+_i$ or $G^-_i$, say $C_2\subseteq G^+_i$. Notice also that $C_2$ is not included in the carrier $N(E_i)$.   Set $G_1:=G^-_i\cup N^+(E_i)$ and $G_2:=G^+_i$.
By Proposition \ref{carriers}, $G_1,G_2,$ and $G_1\cap G_2=N^+(E_i)$ are two-dimensional partial cubes, thus $G$ is a 2d-amalgam of $G_1$ and $G_2$.
Conversely, suppose that a partial cube $G$ is a 2d-amalgam of two-dimensional partial cubes $G_1$ and $G_2$. Consider the canonical COM completions $G_1\urcorner$ and $G_2\urcorner$ of $G_1$ and $G_2$, which are in $\mathcal{F}(Q_3)$ by the Lemma \ref{1extension}.
Then $G_1\urcorner\cap G_2\urcorner$ coincides with $G_{12}\urcorner$. Therefore, by \cite[Proposition 7]{BaChKn} this provides a COM $G'$, which is a COM amalgam of $G_1\urcorner$ and
$G_2\urcorner$ along $G_{12}\urcorner$ without creating new antipodal subgraphs. Using the Lemma \ref{lem:rankofcell}, we deduce that $G'\in {\mathcal F}(Q_3)$. Since the graph $G$ is isometrically embedded in $G'$, $G\in {\mathcal F}(Q_3)$,
which concludes the proof.
\end{proof}

The 2-dimensional cell complex ${\mathbf C}(G)$ of a partial cube $G$ is simply connected but not contractible even if $G$ is two-dimensional. However, for a two-dimensional partial cube $G$
there is a simple remedy: one can consider the {\it (combinatorial) cell complex} having gated cycles and gated full
subdivisions of $G$ as cells. However, since full subdivisions cannot be directly
represented by Euclidean cells, this complex does not have a direct geometric meaning. One possibility is to replace each gated full subdivision $SK_n$ by a regular
Euclidean simplex with sides of length 2 and  each gated cycle  by a regular Euclidean polygon. Denote the resulting polyhedral complex by ${\mathbf X}(G)$. Notice that two cells of
${\mathbf X}(G)$ can intersect in an edge of a polygonal cell or in a half-edge of a simplex. This way, with each two-dimensional partial cube $G$ we associate a polyhedral complex
${\mathbf X}(G)$ which may have cells of arbitrary dimensions. Alternatively, one can associate to $G$ the cell complex ${\mathbf C}(G\urcorner)$ of
the canonical COM completion $G\urcorner$ of $G$. Recall that in ${\mathbf C}(G\urcorner)$, each gated cycle of $G$ is replaced by a regular Euclidean polygon and each gated full subdivision $SK_n$
of $G$ is extended in $G\urcorner$ to $SK^*_n$ and this correspond to a bouquet of squares in ${\mathbf C}(G\urcorner)$. Thus ${\mathbf C}(G\urcorner)$ is a 2-dimensional cell complex.

\begin{corollary} \label{contractible} If $G\in {\mathcal F}(Q_3)$, then the complexes  ${\mathbf X}(G)$ and ${\mathbf C}(G\urcorner)$ are contractible.
\end{corollary}

\begin{proof} That ${\mathbf C}(G\urcorner)$ is contractible follows from the fact that $G\urcorner$ is a two-dimensional COM (Proposition \ref{prop:extendtoCOM}) and that
the cell complexes of COMs are contractible (Proposition 15 of \cite{BaChKn}). The proof that ${\mathbf X}(G)$ is contractible uses the same arguments as the
proof of \cite[Proposition 15]{BaChKn}. We prove the contractibility of ${\mathbf X}(G)$ by induction on the number of maximal cells of ${\mathbf X}(G)$ by using the
gluing lemma \cite[Lemma 10.3]{Bj} and Proposition \ref{carriers}. By the gluing lemma, if $\bf X$ is a cell complex which
is the union of two contractible cell complexes ${\bf X}_1$ and ${\bf X}_2$ such that their intersection ${\bf X}_1\cap {\bf X}_2$ is contractible, then $\bf X$
is contractible. If ${\mathbf X}(G)$ consists of a single maximal cell, then this cell is either a polygon or a simplex, thus is contractible. If ${\mathbf X}(G)$ contains at least two cells, then
by the first assertion of Proposition \ref{amalgam} $G$ is a 2d-amalgam of two-dimensional partial cubes $G_1$ and $G_2$ along a two-dimensional partial cube $G_{12}$.
By induction assumption, the complexes ${\mathbf X}(G_1)$, ${\mathbf X}(G_1)$, and ${\mathbf X}(G_{12})={\mathbf X}(G_1)\cap {\mathbf X}(G_2)$  are contractible, thus ${\mathbf X}(G)$
is contractible by gluing lemma.
\end{proof}


\section{Characterizations  of two-dimensional partial cubes}
The goal of this section is to give a characterization of two-dimensional partial cubes, summarizing all the properties established in the previous sections:

\begin{theorem} \label{characterization}
For a partial cube $G=(V,E)$ the following conditions are equivalent:
\begin{itemize}
\item[(i)] $G$ is a two-dimensional partial cube;
\item[(ii)] the carriers $N(E_i)$ of all $\Theta$-classes of $G$, defined with respect to the cell complex ${\mathbf C}(G)$, are two-dimensional partial cubes;
\item[(iii)] the hyperplanes of $G$ are virtual isometric trees;
\item[(iv)] $G$ can be obtained from the one-vertex graph via a sequence $\{ (G_i^1,G^0_i,G^2_i): i=1,\ldots,m\}$ of isometric expansions,
where each $G^0_i, i=1,\ldots,m$ has VC-dimension $\le 1$;
\item[(v)] $G$ can be obtained via 2d-amalgams from even cycles and full subdivisions;
\item[(vi)] $G$ has an extension to a two-dimensional ample partial cube.
\end{itemize}
Moreover, any two-dimensional partial cube $G$ satisfies the following condition:
\begin{itemize}
\item[(vii)] the gated hull of each isometric cycle of $G$ is a disk or a full subdivision.
\end{itemize}
\end{theorem}

\begin{proof} The implication (i)$\Rightarrow$(ii) is the content of Proposition \ref{carriers}. To prove that (ii)$\Rightarrow$(iii) notice that since $N(E_i)$
is a two-dimensional partial cube, by Propositions \ref{VCdim_pc} and \ref{virtual_isometric_tree} it follows that the hyperplane
of the $\Theta$-class $E_i$ of $N(E_i)$ is a virtual isometric tree.
Since this hyperplane of $N(E_i)$ coincides with the hyperplane $H_i$ of $G$, we deduce that all hyperplanes of $G$ are virtual isometric trees, establishing (ii)$\Rightarrow$(iii).
The implication (iii)$\Rightarrow$(i) follows from Propositions \ref{VCdim_pc} and \ref{virtual_isometric_tree}.
The equivalence  (i)$\Leftrightarrow$(iv) follows
from  Proposition \ref{expansion-Qd+1}. The equivalence (i)$\Leftrightarrow$(v) follows from Proposition \ref{amalgam}.  The implication (i)$\Rightarrow$(vi) follows from Theorem \ref{thm:extendtoample} and the implication
(vi)$\Rightarrow$(i) is evident. Finally, the implication (i)$\Rightarrow$(vii) is the content of Theorem \ref{FS+AOM}.
\end{proof}

Note that it is not true that if in a partial cube $G$ the convex hull of every isometric cycle is in $\mathcal{F}(Q_3)$, then $G\in\mathcal{F}(Q_3)$; see $X_2^4$ in Figure~\ref{fig:COMobstructions}.
However, we conjecture that the condition (vii) of Theorem \ref{characterization} is equivalent to conditions (i)-(vi):

\begin{conjecture} Any partial cube $G$ in which the gated hull of each isometric cycle  is a disk or a full subdivision is two-dimensional.
\end{conjecture}


\section{Final remarks} In this paper, we provided several characterizations of two-dimensional partial cubes via hyperplanes, isometric expansions and amalgamations, cells and carriers, and gated hulls of isometric cycles.
One important feature of such graphs is that gated hulls of isometric cycles have a precise structure: they are either full subdivisions of complete graphs or disks, which are plane graphs representable
as graphs of regions of pseudoline arrangements. Using those results, first we show that any two-dimensional partial cube $G$ can be completed in a canonical way to a COM $G\urcorner$ of rank $2$  and that $G\urcorner$ can be further completed to
an ample partial cube $G^{\top}:=(G\urcorner)\ulcorner$ of VC-dimension $2$. Notice that $G$ is isometrically embedded in $G\urcorner$ and that $G\urcorner$ is isometrically embedded in $G^{\top}$. This answers in the positive (and in the strong way)
the question of \cite{MoWa} for partial cubes of VC-dimension $2$. However, for Theorem~\ref{thm:extendtoample} it is essential that the input is a partial cube:  Figure~\ref{fig:T2} presents a (non-isometric) subgraph $Z$ of $Q_4$ of VC-dimension $2$, such
that any ample partial cube containing $Z$ has VC-dimension $3$. Therefore, it seems to us interesting and nontrivial to solve the question of \cite{RuRuBa} and \cite{MoWa} {\it for all (non-isometric) subgraphs of hypercubes of VC-dimension $2$} (alias,
for arbitrary set families of VC-dimension 2).

\begin{figure}
\centering
\includegraphics[width=.25\textwidth]{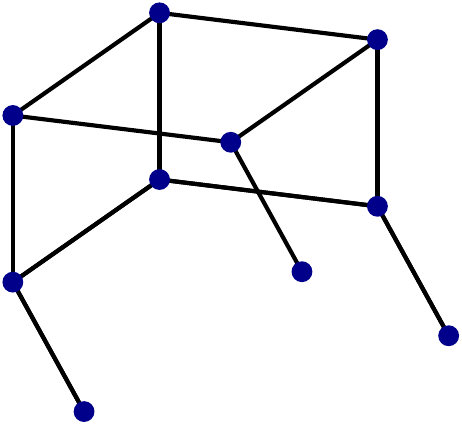}
\caption{A subgraph $Z$ of $Q_4$ of VC-dimension $2$, such that any ample partial cube containing $Z$ has VC-dimension $3$.}
\label{fig:T2}
\end{figure}

It is also important to investigate the completion questions of \cite{MoWa} and \cite{RuRuBa} for all partial cubes from ${\mathcal F}(Q_{d+1})$ (i.e., for partial cubes of VC-dimension $\le d$). For this, it will be interesting to see which results for partial cubes from ${\mathcal F}(Q_3)$ can be extended to graphs from ${\mathcal F}(Q_{d+1})$. We have the impression, that some of the results on disks can be extended to balls; a partial cube is a \emph{$d$-ball} if $G\in\mathcal{F}(Q_{d+1})$ and $G$ contains an isometric antipodal subgraph $C\in\mathcal{F}(Q_{d+1})$ such that $G=\conv(C)$. With this is mind,  one next step would be to study the class $\mathcal{F}(Q_{4})$.



\section*{Acknowledgements}
We are grateful to the anonymous referees for a careful reading  of the paper and numerous useful comments and improvements.
This work was supported by the ANR project DISTANCIA (ANR-17-CE40-0015). The second author moreover was supported by the Spanish \emph{Ministerio de Econom\'ia,
Industria y Competitividad} through grant RYC-2017-22701.

\bibliographystyle{siamplain}
%

\end{document}